\documentclass{m2an}

\usepackage{amssymb,amsfonts,latexsym,amsmath,wrapfig,graphicx}

\usepackage{bm}

\usepackage{enumitem}

\usepackage{chngcntr}
\usepackage{xcolor}
\usepackage{braket}

\usepackage{hyperref}
\hypersetup{
  colorlinks   = true, %Colours links instead of ugly boxes
  urlcolor     = blue, %Colour for external hyperlinks
  linkcolor    = blue, %Colour of internal links
  citecolor   = red %Colour of citations
}

\usepackage[yyyymmdd,hhmmss]{datetime}
\usepackage[all]{background}
\usepackage{soul}

\usepackage{mathrsfs}

\usepackage[capitalize]{cleveref}

\numberwithin{equation}{section}

\theoremstyle{plain} 
\newtheorem{theorem}{Theorem}[section] 
\newtheorem{lemma}[theorem]{Lemma} 
\newtheorem{corollary}[theorem]{Corollary} 
\newtheorem{proposition}[theorem]{Proposition}
\theoremstyle{definition} 
\newtheorem{definition}[theorem]{Definition} 
\newtheorem{remark}[theorem]{Remark}
\newtheorem{example}[theorem]{Example}

\SetBgContents{}% Set contents

\undef\ul

\DeclareMathOperator{\rk}{rk}
\DeclareMathOperator{\Span}{Span}

\DeclareMathOperator{\Ran}{ran}
\DeclareMathOperator{\Real}{Re}
\DeclareMathOperator{\Imag}{Im}

\DeclareMathOperator{\sgn}{sgn}

\newcommand{\RR}{\mathbb{R}}
\newcommand{\CC}{\mathbb{C}}
\newcommand{\NN}{\mathbb{N}}

\newcommand{\RRn}{\mathbb{R}^n}
\newcommand{\CCn}{\mathbb{C}^n}

\newcommand{\VV}{\mathbb{V}}

\newcommand{\HC}{\mathfrak{H}}
\newcommand{\LC}{\mathfrak{L}}

\newcommand{\ol}{\overline}
\newcommand{\ul}{\underline}

\renewcommand{\vec}[1]{\mathbf{#1}}
\newcommand{\dua}[2]{\langle #1, #2 \rangle}
\newcommand{\bdua}[2]{\Big\langle #1, #2 \Big\rangle}

\newcommand{\vdua}[2]{\dua{#1}{#2}}
\newcommand{\hdua}[2]{\dua{#1}{#2}}

\renewcommand{\epsilon}{\varepsilon}
\renewcommand{\kappa}{\varkappa}
\renewcommand{\phi}{\varphi}

\renewcommand{\setminus}{\smallsetminus}

\newcommand{\wt}[1]{\widetilde{#1}}
\newcommand{\wh}[1]{\widehat{#1}}

\newcommand{\grad}{\bm{\nabla}}
\newcommand{\lapl}{\triangle}

\newcommand{\ham}{\mathcal{H}}

\newcommand{\opA}{\mathcal{A}}
\newcommand{\VC}{\mathfrak{V}}

\newcommand{\ene}{\mathcal{E}}

\newcommand{\opK}{\mathcal{K}}

\newcommand{\Fock}{\mathfrak{F}}
\newcommand{\Hil}{\mathfrak{h}}

\newcommand{\vx}{\mathbf{x}}

\newcommand{\rmd}{\mathrm{d}}

\newcommand{\KP}{\mathrm{KP}}

\newcommand{\cc}{\mathrm{CC}}

\newcommand{\full}{\mathrm{full}}

\newcommand{\HF}{\mathrm{HF}}
\newcommand{\LI}{\mathrm{L}}

\newcommand{\vertiii}[1]{{\left\vert\kern-0.25ex\left\vert\kern-0.25ex\left\vert #1
   \right\vert\kern-0.25ex\right\vert\kern-0.25ex\right\vert}}

\DeclareSymbolFont{bbold}{U}{bbold}{m}{n}
\DeclareSymbolFontAlphabet{\mathbbold}{bbold}
\newcommand{\iden}{\mathbbold{1}}

\begin{document}

% Title. If the supplement option is on, then "Supplementary Material"
% is automatically inserted before the title.
\title{Coupled-Cluster Theory Revisited\\\vspace{1em} Part II: Analysis of the single-reference Coupled-Cluster equations}

% Authors: full names plus addresses.
\author{Mih\'aly A. Csirik}\address{Hylleraas Centre for Quantum Molecular Sciences, Department of Chemistry, University of Oslo, P.O. Box 1033 Blindern, N-0315 Oslo, Norway (\email{m.a.csirik@kjemi.uio.no})} 
\author{Andre Laestadius}\address{Department of Computer Science, Oslo Metropolitan University, P.O. Box 4 St. Olavs plass, NO-0130 Oslo, Norway (\email{andre.laestadius@oslomet.no})}\sameaddress{1}

\date{Received June 22, 2022. Accepted November 30, 2022.}

% Sets running headers as well as PDF title and authors
%\headers{Coupled-Cluster Theory}{M. A. Csirik and A. Laestadius}

\begin{abstract}
In a series of two articles, we propose a comprehensive mathematical framework for Coupled-Cluster-type methods. 
In this second part, we analyze the nonlinear equations of the single-reference Coupled-Cluster method using topological degree theory.
We establish existence results and qualitative information about the solutions of these equations that also sheds light of the numerically observed behavior.
In particular, we compute the topological index of the zeros of the single-reference Coupled-Cluster mapping.
For the truncated Coupled-Cluster method, we derive an energy error bound for approximate eigenstates of the Schr\"odinger equation.
\end{abstract}

 

\maketitle

\section{Introduction}

The present article is concerned with the analysis of the single-reference Coupled-Cluster (SRCC) equations and its truncated variants.
It is known that the \emph{truncated} CC equations have a large number of  solutions, some of which exhibit unphysical behavior.
Recall that the truncated CC equations are \emph{not} equivalent to the truncated CI equations (see Section 2.3 of Part I), hence, there is no obvious connection with the Schr\"odinger equation in general---truncated CC solutions can be very far from the desired Full CI solutions.
Also, the SRCC method is typically unreliable for treating degenerate states.

Since the early days of the CC method, researchers in quantum chemistry have been interested in understanding the complicated behavior of the (truncated) CC equations \cite{vzivkovic1977existence,zivkovic1978analytic,paldus1984degeneracy,JankowskiKowalsiJankowski1994,kowalski1998full,jankowski1999correspondence,piecuch2000search,jankowski1999physical2,jankowski1999physical4,kossoski2021excited}. This line of investigation mainly consisted in somehow ``connecting'' the truncated CC solutions to the untruncated ones,
thereby identifying which truncated solutions can be considered ``physical'' and which should be discarded as ``unphysical''. The comprehensive investigation
by K. Kowalski and P. Piecuch \cite{piecuch2000search} classifies solutions based on a certain homotopy. While the universality of this approach is unclear (see \cref{homormrk} below),
\cite{piecuch2000search} demonstrates without doubt that the solutions of the truncated CC equations exhibit rather complicated behavior, which signals the
need for a deeper analytical investigation. Since the CC equations are a system of quartic polynomial equations, the complete set of solutions can be 
computed by numerical means (Kowalski and Piecuch used the HOMPACK software \cite{watson1987algorithm}). However, in our analysis we view the CC equations
as an abstract nonlinear equation and disregard their polynomial structure (this route was also taken by the seminal \cite{schneider2009analysis} and its follow-up works, discussed below). The main reason for this is that system is \emph{very large}, and this seems prohibitive 
for the application of algebraic approaches.

In this article, we analyze the CC equations and the homotopy of Kowalski and Piecuch using a standard tool of nonlinear analysis, topological degree theory. 
In order to do this, we need to restrict ourselves to the finite-dimensional case (\cref{topdegfin}). The present article differs in flavor from the previous 
mathematical investigations, which where more standard ``numerical analysis'' approaches. However, a closer look at the qualitative properties of the
(truncated) CC equations seems necessary, that, in addition to the ground-state also describes excited states.

\subsection{Previous work}\label{secprev}

Similarly to Part I, our approach is based on the analysis of the single-reference CC method by R. Schneider \cite{schneider2009analysis}. 
In that seminal work, a thorough description of the basic building blocks of the method, namely excitation- and cluster operators, and their algebraic and functional-analytic properties are given.
Under certain assumptions,
the CC equations (a nonlinear system of equations consisting of quartic polynomials) are formulated in terms of a locally strongly monotone and locally Lipschitz operator
defined on an appropriate space. Under further hypotheses, this establishes local existence and uniqueness of a (Galerkin projected) ground-state solution of the equation 
and moreover \emph{quasi-optimality} of the projected CC solution (Theorem 5.8 ibid.).
 Perhaps the most important contribution of \cite{schneider2009analysis} is a quadratic energy error estimate (Theorem 6.3 ibid.).
It is worth emphasizing that Schneider's analysis is a local one: \emph{``[...] experience indicates that, in general, it cannot be expected that strong monotonicity always holds, or the constants might be extremely bad. Therefore we expect to get local existence results at best.''} (ibid. p. 30)

Schneider's original analysis was carried out in the finite-dimensional case only. This was remedied in two subsequent articles
by T. Rohwedder \cite{rohwedder2013continuous} and then by both of them \cite{rohwedder2013error}. The article \cite{rohwedder2013continuous} establishes important
technical tools and rigorously proves that the untruncated CC problem is equivalent to the Full CI problem in the infinite-dimensional case, i.e. to essentially the Schr\"odinger equation.
Using the said tools, the subsequent paper \cite{rohwedder2013error} also establishes local uniqueness and existence of a solution to the truncated CC equations
in a neighborhood of the untruncated CC solution, under certain assumptions. Further, \cite{rohwedder2013error} also extends the energy error estimates of \cite{schneider2009analysis}
to the infinite-dimensional case.

This line of investigation was continued by S.~Kvaal and A.~L. in \cite{laestadius2018analysis} for the Extended CC (ECC) method 
based on the ``bivariational principle''~\cite{Arponen1983}. In this case, local strong monotonicity for the ECC mapping can be established so that quasi-optimality follows along similar lines as previously done by Schneider and Rohwedder. In the ECC theory, the traditional CC theory is recovered as a special case.

Furthermore, the local strong monotonicity-based analysis was applied to a variant of the traditional CC method by F. M. Faulstich et~al. \cite{faulstich2019analysis},
namely to the Tailored Coupled-Cluster (TCC) method. The TCC approach splits the computational task of evaluating the ground
state energy into two parts: solving for the statically correlated wave function on a complete active space, %part using CI
and then on top of that %, 
accounting for the dynamical correlation using the CC method. Numerical investigations
based on \cite{faulstich2019analysis} were conducted in~\cite{faulstich2019numerical}.

Finally, we mention the survey article \cite{laestadius2019coupled} for more details
on the use of local strong monotonicity-based methods in the analysis of CC methods.

\subsection{Outline}

While the knowledge of the notations and results of Part I is not strictly necessary for this Part II, some of the results in \cref{secsrcc}
do employ \emph{excitation graphs} to a certain extent, introduced in Section 3.2 of Part I.
 To make the present work as self-contained as possible, we define all the necessary 
concepts without the use of Part I; in other words, in the ``traditional'' second-quantized way.

In \cref{secbg}, we describe the setting of the quantum-mechanical problems the CC theory is aimed at. 
In \cref{sectopdeg} we state a few results that we use from finite-dimensional topological degree theory.
 
The analysis of the single-reference CC (SRCC) method begins in \cref{secsrcc}.
Basic properties of the SRCC mapping are discussed in \cref{ccbasic}. After this, the local properties of the SRCC mapping,
such as strong monotonicity and topological index in both the non-degenerate-, and in the degenerate case are considered in \cref{secsrccop}.
We also look at the complex SRCC mapping in \cref{secsrcccmplx}.

In \cref{seccont}, an important class of homotopies is defined that can be used for proving the existence of a solution to the truncated SRCC mappings.
In \cref{seckp} a homotopy is considered that was invented specifically to connect CC methods of different truncation levels. We prove 
an existence result and calculate the topological index of the homotopy.
Finally, we derive an energy error estimate in \cref{enesec} using the results of \cref{kpapp}.

\section{Background}\label{secbg}

The usual notation $B(a,r)$ is used for the open ball of radius $r$ and center $a$, also $B^*(a,r)=B(a,r)\setminus\{a\}$
denotes the punctured ball. 

 The spectrum of a linear operator $A$ is written $\sigma(A)$,
the elements of its discrete spectrum as $\ene_n(A)$, where $n=0,1,2,\ldots$, if $A$ is bounded from below.
We use the usual notation $[A,B]=AB-BA$ for the commutator. The (conjugate) transpose of $A$ is denoted as $A^\dag$. For normed spaces $V$ and $W$, the symbol $\mathcal{L}(V,W)$ denotes normed
space of \emph{bounded} linear mappings $V\to W$ endowed with the operator norm $\|\cdot\|_{\mathcal{L}(V,W)}$. Furthermore, $V^*$ denotes the (continuous) dual space. As usual, $[a,b]$ denotes the (closed) line segment between $a,b\in V$.

\subsection{Second quantization}

In this section, we first briefly review the framework of second quantization. The notations mainly follow \cite{lewin2011geometric}
and \cite{solovej2007many}. Next, we formulate the Schr\"odinger Hamiltonian in second quantization. 
Finally, we introduce the excitation-, and cluster operators and cluster amplitude spaces through the use of 
creation-, and annihilation operators.

\subsubsection{Fermionic Fock space} For simplicity, and because the Hamiltonian we will be considering is spin-independent, we neglect spin and consider
$\Hil=L^2(\RR^3)$ as the \emph{one-particle Hilbert space}. The tensor powers of $\Hil$ are denoted by $\Hil^N=\bigotimes^N\Hil$ and antisymmetric powers of $\Hil$ are denoted by 
$\Hil_a^N=\bigwedge^N\Hil$ for any $N\ge 0$,
where we introduced the notation $\Hil^0=\CC$ and $\Hil_a^0=\CC$. The \emph{fermionic Fock space} is then the Hilbert space
given by the infinite direct sum
$\Fock_a:=\bigoplus_{N\ge 0} \Hil_a^N=\Hil_a^0\oplus\Hil_a^1\oplus\Hil_a^2\oplus\ldots$,
in other words
$$
\Fock_a=\Big\{ \Psi=(\psi^0,\psi^1,\ldots) : \psi^N\in\Hil_a^N\,(N\ge 0),\; \|\Psi\|_\Fock^2:=\sum_{N\ge 0} \|\psi^N\|_{\Hil^N}^2<\infty \Big\},
$$
endowed with the inner product $\dua{\Psi_1}{\Psi_2}_{\Fock}=\sum_{N\ge 0} \dua{\psi_1^N}{\psi_2^N}_{\Hil^N}$,
for any $\Psi_1=(\psi_1^0,\psi_1^1,\ldots)\in\Fock_a$ and $\Psi_2=(\psi_2^0,\psi_2^1,\ldots)\in\Fock_a$.
The \emph{vacuum state} is the distinguished element $\Omega=(1,0,\ldots)\in\Fock_a$. The space $\Hil_a^N$ can be identified as a subspace 
of $\Fock_a$, in this context it is called the \emph{$N$-particle sector of $\Fock$}.
For $\Psi_1\in \Hil_a^{N_1}$ and $\Psi_2\in\Hil_a^{N_2}$ 
define $\Psi_1\wedge\Psi_2\in\Hil_a^{N_1+N_2}$ via
$$
\Psi_1\wedge\Psi_2(\vx_1,\ldots,\vx_{N_1+N_2})=C_{N_1,N_2} \sum_{\sigma\in\mathfrak{S}_{N_1+N_2}} 
(\sgn\,\sigma) \Psi_1(\vx_{\sigma(1)},\ldots,\vx_{\sigma(N_1)}) \Psi_2(\vx_{\sigma(N_1+1)}, \ldots, \vx_{\sigma(N_1+N_2)}),
$$
with the normalization constant $C_{N_1,N_2}=(N_1!N_2!(N_1+N_2)!)^{-1/2}$. Here, $\mathfrak{S}_n$ denotes the
permutation group of $\{1,\ldots,n\}$ and $\sgn\sigma$ the sign of the permutation $\sigma$. 

Let $\{\phi_p\}_{p=1}^\infty\subset\Hil$ be an orthonormal basis and define the \emph{Slater determinant}
$\Phi_\alpha:=\phi_{\alpha_1}\wedge\ldots\wedge\phi_{\alpha_N}\in\Hil_a^N$ for every multiindex $\alpha\in\NN_0^N$ with $\alpha_1<\ldots<\alpha_N$.
Then $\{\Phi_\alpha\}_{\alpha}$ is an orthonormal basis of $\Hil_a^N$.

\subsubsection{Creation-, and annihilation operators}\label{creannihsec}

For fixed $\phi\in\Hil$, the (fermionic) \emph{creation operator} $a^\dag(\phi) : \Fock_a\to\Fock_a$ is defined as $a^\dag(\phi)\Psi=\phi\wedge\Psi$
for any $\Psi\in\Hil_a^N$ and extended boundedly and linearly to the whole space. Also, for fixed $\phi\in\Hil$,
define the (fermionic) \emph{annihilation operator} $a(\phi) : \Fock_a\to\Fock_a$ as the adjoint of $a^\dag(\phi)$, i.e.
$\dua{\Psi_1}{a^\dag(\phi)\Psi_2}_{\Fock}=\dua{a(\phi)\Psi_1}{\Psi_2}_{\Fock}$ for all $\Psi_1,\Psi_2\in\Fock_a$.
It is an easy calculation to show that $(a(\phi)\Psi)(\vx_1,\ldots,\vx_{N-1})=N^{-1/2}\int \phi(\vec{x})\Psi(\vx,\vx_1,\ldots,\vx_{N-1})\,\rmd \vx$
for all $\Psi\in\Hil_a^N$ ($N\ge 1$) and $a(\phi)\Omega=0$. The \emph{canonical anticommutation relations (CAR)}
$$
\begin{aligned}
a^\dag(\phi) a(\phi') + a(\phi') a^\dag(\phi) &= \dua{\phi}{\phi'}_{\Hil} I_{\Fock_a}, \\
a^\dag(\phi) a^\dag(\phi') + a^\dag(\phi')a^\dag(\phi) &= 0,\\
a(\phi) a(\phi') + a(\phi')a(\phi) &=0,
\end{aligned}
$$
hold true for any $\phi,\phi'\in\Hil$, where $I_{\Fock_a}:\Fock_a\to\Fock_a$ denotes the identity map on $\Fock_a$. Since both $a^\dag(\phi)$ and $a(\phi)$ are linear and bounded in $\phi$, it is enough to specify them on an orthonormal
basis $\{\phi_p\}_{p\ge 1}\subset\Hil$, i.e. $a_p^\dag:=a^\dag(\phi_p)$ and $a_p:=a(\phi_p)$ ($p\ge 1$) completely determines the families
$\{a^\dag(\phi)\}_{\phi\in\Hil}$ and $\{a(\phi)\}_{\phi\in\Hil}$.
It is then clear that any Slater determinant $\Phi_\alpha\in\Hil_a^N$ can be ``created'' from the vacuum state as
$\Phi_\alpha= a_{\alpha_1}^\dag \cdots a_{\alpha_N}^\dag \Omega$.

\subsubsection{First-, and second quantization of one-, and two-body operators}

Let $h:\Hil\to\Hil$ be a linear operator. Define its \emph{first quantization} as $\sum_{1\le j\le N} h_j : \Hil_a^N\to\Hil_a^N$
where 
$$
h_j=I\otimes\ldots\otimes I \otimes \underbrace{h}_{\text{$j$th}}\otimes I\otimes \ldots\otimes I.
$$
Here, $\iden$ denotes the identity on $\Hil$.
The \emph{second quantization} of $h$ is defined as $0\oplus \bigoplus_{N\ge 1} \sum_{1\le j\le N} h_j : \Fock_a\to\Fock_a$.

Analogously, one can define the first quantization of a two-body operator $W:\Hil_a^2 \to \Hil_a^2$ as $\sum_{1\le i,j\le N} W_{ij} : \Hil_a^N\to\Hil_a^N$,
where
$$
W_{ij}=I\otimes\ldots\otimes I \otimes \underbrace{W}_{\text{$i$th}} \otimes I \ldots I \otimes \underbrace{W}_{\text{$j$th}}\otimes I\otimes \ldots\otimes I.
$$
The second quantization of $W$ is
$0\oplus 0\oplus \bigoplus_{N\ge 2} \sum_{1\le i<j\le N} W_{ij} : \Fock_a\to\Fock_a$.
Both the one-, and the two-body operators can be expressed using creation-, and annihilation operators as follows.
\begin{theorem}\label{sqonebody}
Let $\{\phi_p\}_{p\ge 1}\subset\Hil$ be an orthonormal basis. Then
$$
0\oplus \bigoplus_{N\ge 1} \sum_{1\le j\le N} h_j =\sum_{p,q\ge 1} \dua{h\phi_p}{\phi_q} a_p^\dag a_q.
$$
and
$$
0\oplus 0\oplus \bigoplus_{N\ge 2} \sum_{1\le i<j\le N} W_{ij} =\sum_{\substack{1\le p\le q\\ 1\le r\le s}} \dua{W\phi_p\wedge\phi_q}{\phi_r\wedge\phi_s} a_p^\dag a_q^\dag a_r a_s.
$$
\end{theorem}

\subsection{Schr\"odinger Hamiltonian in second-quantized form}\label{hamsec}

For convenience and in accordance to Part I, we define the $N$-particle fermionic spaces
$$
\LC^2:=\Hil_a^N=\bigwedge^N L^2(\RR^{3}), \quad \text{and} \quad \HC^1:=\LC^2 \cap H^1(\RR^{3N}),
$$
endowed with usual inner products $\dua{\cdot}{\cdot}$ and $\dua{\cdot}{\cdot}_{\HC^1}$, and norms $\|\cdot\|$ and $\|\cdot\|_{\HC^1}$, see Section 2.1 in Part I. 
We also set $\HC^2:=\LC^2 \cap H^2(\RR^{3N})$, and $\HC^{-1}:=(\HC^1)^*$ as usual. Henceforth, we employ the same convention as explained after Remark 2.1 in Part I.

Recall (see Section 2.2 of Part I) the Schr\"odinger Hamiltonian $\ham:\LC^2\to\LC^2$, which is the self-adjoint operator
$$
\ham=\sum_{1\le i\le N} \Bigg[-\frac{1}{2}\lapl_{\vec{x}_i} + V(\vec{x}_i)\Bigg] + \sum_{1\le i<j\le N} w(\vec{x}_i-\vec{x}_j),
$$
with domain $D(\ham)=\HC^2$, where $V,w:\RR^3\to\RR$ are Kato class potentials: $V,w\in L^{3/2}(\RR^3)+L^\infty_\epsilon(\RR^3)$ and $w$ is even.
The quadratic form corresponding to $\ham$ is denoted as $\ene(\Psi)=\dua{\ham\Psi}{\Psi}$, and has form domain $Q(\ene)=\HC^1$.
Recall that there is a constant $M>0$, such that
\begin{equation}\label{Hbound}
\dua{\ham\Psi}{\Phi}\le M\|\Psi\|_{\HC^1}\|\Phi\|_{\HC^1}
\end{equation}
for all $\Psi,\Phi\in\HC^1$.
Thus, $\ham$ can be extended to a bounded mapping $\HC^1\to\HC^{-1}$, which we denote with the same symbol.
We say that $\Psi\in\HC^1$ and $\ene\in\RR$ satisfy the \emph{weak Schr\"odinger equation} if  $\dua{\ham\Psi}{\Phi}=\ene\dua{\Psi}{\Phi}$ for all $\Phi\in\HC^1$.

According to \cref{sqonebody}, $\ham$ can also be given in the second quantized form 
$$
\ham=\sum_{p,q\ge 1} h_{pq} a_p^\dag a_q + \sum_{\substack{1\le p\le q\\ 1\le r\le s}} W_{pq,rs} a_p^\dag a_q^\dag a_r a_s,
$$
where $h_{pq}=\frac{1}{2}\dua{\grad\phi_p}{\grad\phi_q} + \dua{V\phi_p}{\phi_q}$
and $W_{pq,rs}=\dua{W\phi_p\wedge\phi_q}{\phi_r\wedge\phi_s}$.

\subsubsection{Truncation of the orbital basis}\label{truncsec}
Our analysis of the CC method is restricted to the finite-dimensional case due to reasons 
described later (\cref{topdegfin}). Let $K\geq N$ and assume that an $L^2$-orthonormal \emph{(spin-)orbital set} $\{\phi_p\}_{1\le p\le K}\subset H^1(\RR^3)$ is given.
We define
$$
\mathfrak{B}_K=\{ \Phi_{\alpha} \in \HC^1 : 1\le\alpha_1<\ldots<\alpha_N\le K \},
$$
as in Section 2.1 of Part I. Then $\mathfrak{B}_K$ is $\LC^2$-orthonormal. Set
$$
\HC^1_K=\Span\mathfrak{B}_K\subset\HC^1. %\LC^2_K=\Span\mathfrak{B}\subset\LC^2,  \quad \text{and} \quad 
$$
We define $\ham_K:\HC_K^1\to(\HC_K^1)^*$ to be the projection of $\ham:\HC^1\to\HC^{-1}$ onto the finite-dimensional subspace $\HC_K^1\subset \HC^1$, i.e. 
$\dua{\ham_K\Phi}{\Psi}=\dua{\ham\Phi}{\Psi}$ for all  $\Phi,\Psi\in\HC_K^1$.
Then $\Psi\in\HC^1_K$ and $\ene\in\RR$ are said to satisfy the \emph{projected Schr\"odinger equation} if 
$$
\dua{\ham_K\Psi}{\Phi}=\ene\dua{\Psi}{\Phi}\quad\text{for all}\quad \Phi\in\HC_K^1.
$$
For a convergence theory for this eigenvalue problem, see e.g. \cite[Theorem 5.11]{yserentant2010regularity}. Vaguely speaking, if $\bigcup_{K\ge N}\{\phi_p\}_{1\le p\le K}$
is dense in $H^1(\RR^3)$, then convergence of the projected eigenvalues to the discrete spectrum of $\ham$ is guaranteed as $K\to\infty$ (see \cite{schneider2009analysis}).

From the computational standpoint, orbital basis truncation is always necessary when solving the Schr\"odinger equation. Of course, the spectral structure
changes (the essential spectrum disappears in finite dimensions), and this means that we replaced the original problem with one having different qualitative properties.
However, the CI and CC methods are mainly used to approximate the ground-state-, or the first few excited energies (and other quantities related to the lowest eigenstates).
With a sufficiently good choice the orbitals $\{\phi_p\}_{1\le p\le K}$, these energies can be approximated rather well in a lot of situations.
Therefore, it is safe to regard the rest of the projected spectrum as ``junk'', and simply ignore it. 
We can conclude that the restriction to finite dimensions allows for a meaningful theory that is still relevant to CC methods.

\subsection{The Hartree--Fock method}\label{hfsec}

In practice, the CC method usually takes the HF orbitals as an input and builds up correction based on them, as described in Section 2.3 of Part I.
Here, we collect the basic facts about the HF method that will be used later on. 

The HF method\footnote{a.k.a. Self-Consistent Field (SCF) method} is based on the minimiziation of the energy over Slater determinants. It is fairly easy to see by direct calculation that the energy 
$\ene(\Phi)=\dua{\ham\Phi}{\Phi}$ of a Slater determinant $\Phi=\phi_1\wedge\ldots\wedge\phi_N\in\HC^1$, with $\{\phi_j\}_{j=1}^N\subset H^1(\RR^3)$ being $L^2$-orthonormal is given by
\begin{align*}
\ene(\phi_1,\ldots,\phi_N):=\ene(\Phi)&=\frac{1}{2}\sum_{i=1}^N \int_{\RR^3} |\grad \phi_i|^2 + \sum_{i=1}^N \int_{\RR^3} V|\phi_i|^2\\
&+\frac{1}{2} \iint_{\RR^3\times\RR^3}  \Bigg[ \sum_{i,j=1}^N |\phi_i(\vec{x})|^2 |\phi_j(\vec{y})|^2
-\Bigg|\sum_{i=1}^N \phi_i(\vec{x})\ol{\phi_i(\vec{y})}\Bigg|^2\Bigg]w(\vec{x}-\vec{y})\,\rmd\vec{x} \rmd\vec{y},
\end{align*}
see e.g. \cite{cances2003computational}.
Hence the task is to determine the \emph{HF energy}
\begin{equation}\label{EHF}
\ene_{\HF} = \ene(\Phi_\HF)=\min_{\substack{\phi_1,\ldots,\phi_N\in H^1(\RR^3)\\\dua{\phi_i}{\phi_j}=\delta_{ij}}} \ene(\phi_1,\ldots,\phi_N),
\end{equation}
along with a \emph{HF minimizer} (or \emph{HF determinant}) $\Phi_\HF$. 
The existence of a HF minimizer $\Phi_\HF$ is guaranteed for the case of positive ions and neutral atoms and molecules (corresponding to the electronic molecular Hamiltonian).
\begin{theorem}\cite{lieb1974solutions}
If $N<Z+1$, then there exists a minimizer $\Phi_\HF$ to \cref{EHF}.
\end{theorem}

Following the seminal work \cite{lieb1974solutions}, much more has been discovered about the mathematical structure of the HF energy functional \cite{lions1987solutions,friesecke2003multiconfiguration,lewin2018existence}.
As usual, a minimizer satisfies the corresponding Euler--Lagrange equations (which are called \emph{Hartree-Fock equations} in this context).
In practice, it is this system of nonlinear integro-differential equations which is discretized and solved.
To describe the HF equations, we make a definition that will be convenient for later purposes.
Fix $\Phi=\phi_1\wedge\ldots\wedge\phi_N$, where $\{\phi_p\}_{p=1}^N\subset H^1(\RR^3)$ is $L^2$-orthonormal. Define the lower semibounded, self-adjoint operator $F_\Phi : L^2(\RR^3)\to L^2(\RR^3)$
with domain $D(F_\Phi)=H^2(\RR^3)$ via the instruction
\begin{align*}
(F_\Phi \psi)(\vec{x}) &= -\frac 1 2 \lapl\psi(\vec{x}) + V(\vec{x})\psi(\vec{x}) + \left(\sum_{i=1}^N \int_{\RR^3} w(\vec{x}-\vec{y})|\phi_i(\vec{y})|^2\right) \psi(\vec{x}) - \sum_{i=1}^N \left(\int_{\RR^3} w(\vec{x}-\vec{y}) \ol{\phi_i(\vec{y})} \psi(\vec{y})\, \rmd \vec y\right) \phi_i(\vec{x})
\end{align*}
for all $\psi\in D(F_\Phi)$ and all $\vec{x}\in\RR^3$. The operator $F_\Phi$ is called the \emph{mean-field operator}.\footnote{It is sometimes
called the Fock operator, but we reserve that name for its $N$-particle version, i.e. its first quantization.} The form domain of $F_\Phi$ is $H^1(\RR^3)$.
The essential spectrum of $F_\Phi$ is $[0,+\infty)$. We summarize the basic properties of the mean-field operator in the next theorem. Let
$$
\mu_n(F_\Phi)=\min_{\substack{\psi_1,\ldots,\psi_n\in H^1(\RR^3)\\\dua{\phi_i}{\phi_j}=\delta_{ij}}}\; \max_{\substack{\psi\in\Span\{\psi_1,\ldots,\psi_n\} \\\|\psi\|=1}} \dua{F_\Phi\psi}{\psi}
$$
denote the min-max values of $F_\Phi$.
\begin{theorem}
Assume that there exists a HF minimizer $\Phi_\HF=\phi_1\wedge\ldots\wedge\phi_N$.
\begin{enumerate}[label=\normalfont(\roman*)]
\item (Hartree-Fock equations) There exists a unitary matrix $\vec{U}\in\CC^{N\times N}$ so that 
$\wt{\phi}_i$ are eigenfunctions of $F_\Phi$ corresponding to its $N$ lowest eigenvalues $\lambda_1\le\ldots\le\lambda_N$,
\begin{equation}\label{hfeq}
F_\Phi \wt{\phi}_j = \lambda_j \wt{\phi}_j, \quad \text{for all} \quad j=1,\ldots,N,
\end{equation}
where $(\wt{\phi}_1,\ldots,\wt{\phi}_N)=\vec{U}(\phi_1,\ldots,\phi_N)$.
\item (Aufbau principle) If $\mu_{N+1}(F_\Phi)$ is an eigenvalue of $F_\Phi$, then
$\lambda_N=\mu_N(F_\Phi)<\mu_{N+1}(F_\Phi)\le 0$.
\end{enumerate}
\end{theorem}
\begin{proof}
See, for example \cite{benedikter2017hartree}.
\end{proof}

The eigenvalue $\lambda_{N}$ corresponds to the \emph{highest occupied molecular orbital} (HOMO) and $\lambda_{N+1}$ to the \emph{lowest unoccupied molecular orbital} (LUMO).
Their difference, $\epsilon_\mathrm{min} :=\lambda_{N+1}-\lambda_N$ is called the \emph{HOMO-LUMO gap}, which is an important quantity in quantum chemistry \cite{bach1997there}.

The first quantization of the mean-field operator $F_\Phi$ is called the \emph{Fock operator} $\mathcal{F}_\Phi : \LC^2\to \LC^2$ 
with domain $D(\mathcal{F}_\Phi)=\bigwedge^N D(F_\Phi)=\HC^2$,
$$
\mathcal{F}_\Phi = F_\Phi \otimes I \otimes \ldots \otimes I + \ldots + I\otimes\ldots\otimes I \otimes F_\Phi.
$$
Henceforth we omit $\Phi$ from the notation, and let $\mathcal{F}:=\mathcal{F}_\Phi$ whenever $\Phi$ is clear from the context.
It is immediate that the HF determinant $\Phi_0:=\Phi_\HF$ is an eigenfunction of $\mathcal{F}$,
$$
\mathcal{F}\Phi_0=\Lambda_0\Phi_0, \quad\text{with}\quad \Lambda_0=\sum_{i=1}^N \lambda_i.
$$
The Fock operator gives rise to a splitting of the molecular Hamiltonian 
\begin{equation}\label{fockfluc}
\mathcal{H}=\mathcal{F} + \mathcal{W}, \quad\text{where}\quad \mathcal{W}:=\ham-\mathcal{F}
\end{equation}
is called the \emph{fluctuation operator}.

For the rest of the section, we consider the finite-dimensional case. 
In practice, the Galerkin projection of the Hartree--Fock equations \cref{hfeq} are solved to obtain the orbitals $\{\phi_i\}_{i=1}^N$. Since the
mean-field operator $F_\Phi$ is self-adjoint, its eigenfunctions can be used to extend these orbitals to an orthonormal basis $\{\phi_i\}_{i=1}^K\subset H^1_K(\RR^3)$. 
Similarly to $\ham_K$ we introduce the projected versions of $\mathcal{F}$ and $\mathcal{W}$,
denoted as $\mathcal{F}_K$ and $\mathcal{W}_K$.
In the orbital basis $\{\phi_i\}_{i=1}^K$, the Fock operator takes the diagonal form $\mathcal{F}_K=\sum_{i=1}^K \lambda_i a_i^\dag a_i$.

Furthermore, if  $\Phi_\alpha=\phi_{\alpha_1}\wedge\ldots\wedge\phi_{\alpha_N}$ with $\alpha_1<\ldots<\alpha_N$ that is obtained from
$\Phi_0=\phi_1\wedge\ldots\wedge\phi_N$ by swapping $r$ orbitals $\phi_{I_j}$ with $\phi_{A_j}$ where $I_j\in\{1,\ldots,N\}$, $A_j\not\in\{1,\ldots,N\}$ ($j=1,\ldots,r\le N$),
then
\begin{equation}\label{fockeig}
\mathcal{F}_K\Phi_\alpha=\Lambda_\alpha\Phi_\alpha, \quad\text{with}\quad \Lambda_\alpha=\sum_{i=1}^N \lambda_{\alpha_i}=\Lambda_0 + \epsilon_{\alpha},
\quad\epsilon_\alpha=\sum_{j=1}^r \lambda_{A_j}-\lambda_{I_j}.
\end{equation}
We would like to emphasize that in the finite-dimensional case, it is possible to have $\lambda_N\le 0<\lambda_{N+1}$.\footnote{The addition of diffuse functions
of the form $e^{-\beta|\mathbf{x}|^2}$ for $\beta\ll 1$ can significantly reduce the HOMO-LUMO gap $\epsilon_{\min}$.}

Unfortunately, in the infinite-dimensional case, it might not be possible to construct a complete eigenbasis $\{\phi_i\}_{i=1}^\infty\subset H^1(\RR^3)$ 
for the mean-field operator $F_\Phi$.

\subsection{Excitation-, and cluster operators}\label{exopsec}

We now briefly recall the basic properties excitation-, and cluster operators using the second-quantized language. 
Alternatively, the reader may consult Section 3.2 of Part I. Here, we introduce the concepts in arbitrary dimensions for completeness, although we
only need the finite-dimensional case in the sequel.

Suppose that $\Phi_0:=\Phi_\HF=\phi_1\wedge\ldots\wedge\phi_N$ is a HF minimizer for some orthonormal $\{\phi_p\}_{p=1}^N$, as discussed in the preceding section. 
The wavefunction $\Phi_0$ is called the \emph{reference determinant} or simply the \emph{reference}.
The functions $\{\phi_p\}_{p=1}^N$ are called \emph{occupied orbitals} and are extended to a complete orthonormal basis of $H^1_K(\RR^3)$
with the \emph{virtual orbitals} $\{\phi_p\}_{p=N+1}^K$ (here, $K=\infty$ is allowed). 
According to \cref{creannihsec}, any Slater determinant $\Phi_\alpha$ may be obtained from the Fock vacuum via a string of creation operators,
in particular $\Phi_0=a_1^\dag\cdots a_N^\dag\Omega$. Here, and henceforth, the symbol $0$ denotes $(1,\ldots,N)$ as multiindex.

As in the preceding section, we may partition the indices $\alpha_1<\ldots<\alpha_N$ into occupied-, and virtual ones
by letting $I_{r+1}<\ldots<I_N$ be the $\alpha_j$'s which lie in $\{1,\ldots,N\}$ and $A_1<\ldots<A_r$ be the $\alpha_j$'s which lie in $\{N+1,N+2,\ldots\}$, where
$r$ is the number of virtual indices in $\alpha$, also called the \emph{rank of $\alpha$}. Now let $I_1<\ldots<I_r$ be given by $\{I_1,\ldots,I_r\}=\{1,\ldots,N\}\setminus \{I_{r+1},\ldots,I_N\}$, then we can write
$$
\Phi_\alpha= a_{A_1}^\dag a_{I_1} \cdots a_{A_r}^\dag a_{I_r} \Phi_0, 
$$
since the string of creation and annihilation operators acting on $\Phi_0$ clearly swaps the occupied orbital $\phi_{I_j}$ with $\phi_{A_j}$ in $\Phi_0$ for all $j=1,\ldots,r$,
yielding precisely $\Phi_\alpha$. Then $X_\alpha=a_{A_1}^\dag a_{I_1} \cdots a_{A_r}^\dag a_{I_r}$ is called an \emph{excitation operator}.
We say that the Slater determinant $\Phi_\alpha$ (or the multiindex $\alpha$), is \emph{singly-} (S), \emph{doubly-} (D), \emph{triply-} (T) etc. \emph{excited} if the rank of $\alpha$ is 1, 2, 3, etc.

This way we constructed a family of bounded operators $X_\alpha$, where $\alpha$ runs over all possible \emph{nonzero} multiindices.
The adjoint of an excitation operator $X_\alpha$, denoted $X_\alpha^\dag$, is called a \emph{de-excitation operator}.
Notice that both excitation-, and de-excitation operators conserve the particle number.

\begin{theorem}
The following properties hold true.
\begin{itemize}
\item[(i)] The product of (de-)excitation operators is an (de-)excitation operator, or, zero. 
\item[(ii)] (commutativity) $X_\alpha X_\beta=X_\beta X_\alpha$ and $X_\alpha^\dag X_\beta^\dag=X_\beta^\dag X_\alpha^\dag$.
\item[(iii)] (nilpotency) $X_\alpha^2=0$ and $(X_\alpha^\dag)^2=0$.
\end{itemize}
\end{theorem}

Linear combinations of excitation operators are called \emph{cluster operators}, i.e. operators of the form $C=\sum_{\alpha\neq 0} c_\alpha X_\alpha$.
Clearly, there is a one-to-one correspondence between $C$ and the so-called \emph{cluster amplitude} $(c_\alpha)_{\alpha\neq 0}$. We use the standard convention
that capital letters $C,T,S,\ldots$ denote cluster operators and small letters $c,t,s,\ldots$ their corresponding (unique) cluster amplitudes.
Clearly, there is a one-to-one correspondence between
functions $\Psi\in\HC^1$ with $\dua{\Psi}{\Phi_0}=0$ and the cluster operators $C_\Psi$ defined as
\begin{equation}\label{clusterop}
C_\Psi=\sum_{\alpha\neq 0} c_\alpha X_\alpha, \quad\text{where}\quad c_\alpha=\dua{\Psi}{\Phi_\alpha}.
\end{equation}

\begin{theorem}\cite{rohwedder2013continuous}
Fix $\Psi\in\HC^1$, such that $\dua{\Psi}{\Phi_0}=0$. Then the following properties hold true.
\begin{itemize}
\item[(i)] The vector space of cluster operators extended with the identity is a closed, commutative, nilpotent subalgebra of $\mathcal{L}(\HC^1)$. In particular, $T^{N+1}=0$ for any $T$.
\item[(ii)] The cluster operator $C_\Psi$ satisfies $C_\Psi\in \mathcal{L}(\HC^1,\HC^1)$.
Furthermore, there is a constant $b>0$ independent of $\Psi$ such that
$\|\Psi\|_{\HC^1}\le \|C_\Psi\|_{\mathcal{L}(\HC^1,\HC^1)}\le b\|\Psi\|_{\HC^1}$.
\item[(iii)] $C_\Psi^\dag\in \mathcal{L}(\HC^1,\HC^1)$, and there is a constant $b'>0$ independent of $\Psi$ such that
$\|C_\Psi^\dag\|_{\mathcal{L}(\HC^1,\HC^1)}\le b'\|\Psi\|_{\HC^1}$,
and there cannot be a uniform lower bound in terms of $\|\Psi\|_{\HC^1}$.
\item[(iv)] $C_\Psi$ can be extended to $\mathcal{L}(\HC^{-1},\HC^{-1})$.
\end{itemize}
\end{theorem}

According to \cite[Lemma 5.2]{rohwedder2013continuous}, for any cluster operator $C\in \mathcal{L}(\HC^1,\HC^1)$, 
there exists a cluster operator $T\in \mathcal{L}(\HC^1,\HC^1)$, such that $I+C=e^T$. This $T$ is uniquely given by $\log(I+C)$.
We would like to point out, that the exponential representation is only true in the \emph{untruncated case}, i.e. when \emph{all} cluster amplitudes $(t_\alpha)$
and all excitation operators $X_\alpha$ are considered. If some truncation is in effect (see below), the set $\{e^T\Phi_0\}$ is unknown.

\subsection{Cluster amplitude spaces}\label{clusampsec}

As mentioned in the previous section, the linear combination coefficients $t=(t_\alpha)$ of a cluster operator expansion $T=\sum_{\alpha\neq 0} t_\alpha X_\alpha$
are called cluster amplitudes. If some \emph{truncation scheme} is present, some of the $t_\alpha$'s and $X_\alpha$'s will be missing (again we refer to Section 3.5 of Part I for the
precise formulation). This truncation procedure is crucial for a feasible implementation of the CC method, and most commonly only the singly- (S), or singly and doubly- (SD),
or the singly, doubly and triply (SDT) excited amplitudes are kept. 

The vector space of (possibly truncated) cluster amplitudes form a subspace of $\ell^2$.
We define the \emph{(cluster) amplitude space}
$$
\VV=\{ t\in\ell^2 : \|T\Phi\|_{\HC^1}<\infty \},
$$
endowed with the inner product $\dua{t}{s}_\VV=\dua{T\Phi_0}{S\Phi_0}_{\HC^1}$. Analogously to $\HC^1\subset\LC^2\subset\HC^{-1}$, the spaces $\VV\subset\ell^2\subset\VV^*$
form a Gelfand triple, so Remark 2.1 of Part I applies. To $\VV$, there corresponds the \emph{functional amplitude space}
$$
\VC=\{ T\Phi_0\in\HC^1 : t\in\VV\}.
$$
In accordance to the concepts developed in Part I, we denote the \emph{full}, i.e. untruncated amplitude space and the corresponding functional amplitude space by $\VV(G^\full)$ and
$\VC(G^\full)$, respectively.

\begin{remark}\label{focknorm}
In the case $K<\infty$, and assuming that HOMO-LUMO gap satisfies $\epsilon_{\min}>0$ (no matter how small), following \cite{schneider2009analysis}, we introduce the norm
$$
\vertiii{t}^2:=\sum_{\alpha\neq 0} \epsilon_\alpha |t_\alpha|^2\quad\text{for all}\quad t\in\VV,
$$
where $\epsilon_\alpha$ denotes the eigenvalues of the Fock operator $\mathcal{F}_K$, see \cref{fockeig}. It was shown in \cite{schneider2009analysis} that
under certain assumptions, the constants in the norm equivalence $\vertiii{\cdot}\sim\|\cdot\|_\VV$
are independent of $K$. Finally, we also define the norm $\vertiii{\cdot}$ on $\VC$ via $\vertiii{T\Phi_0}:=\vertiii{t}$ for any $t\in \VV$. 
\end{remark}

\section{Topological degree theory}\label{sectopdeg}

In this section, we give a short introduction to the basic notions of \emph{finite-dimensional} topological degree theory, and state the results
that we will need in the forthcoming analysis of the CC method. While topological degree theory is well-known in
the field of nonlinear analysis and dynamical systems, it is not so ubiquitous in mathematical physics 
(a nice application is furnished by Ginzburg--Landau theory, see e.g. \cite[Section 2.4]{dinca2009brouwer}).
Therefore, an inclusion of a brief summary of this important and powerful tool in the present article seems justified.

Here, we mainly follow the set of notes by Dinca and Mawhin \cite{dinca2009brouwer} (which was released as a book recently \cite{dinca2021brouwer}),
 where the proofs may be found. The concept of the topological (a.k.a. mapping-, or Brouwer-) degree of a mapping goes back to Kronecker, Poincar\'e and Brouwer;
  Leray and Schauder extended the concept to infinite dimensions and demonstrated its usefulness in the theory of partial differential equations.
  (See \cite[Chapter 23]{dinca2009brouwer} for a fascinating historical perspective.)
The standard textbooks on the topic are \cite{cronin1995fixed,zeidler1993nonlinear,cho2006topological}.
 
\begin{remark}\label{topdegfin}
Unfortunately, there cannot be a ``topological degree theory'' for general continuous mappings in infinite dimensions. 
The reason for this, roughly speaking, is that such a theory would imply the Brouwer fixed point theorem, but that does not
hold in Hilbert space as pointed out by Kakutani \cite[Example 5.1.7]{drabek2007methods}. 
There are many generalizations of topological degree theory to infinite dimensions, all requiring some stringent 
assumptions on the class of mappings considered \cite{drabek2007methods,cho2006topological,skrypnik1994methods,petryshyn1992approximation}.
\end{remark}
 
\subsection{Construction and basic properties}
The first result gives an axiomatic characterization of the topological degree as the \emph{unique} additive homotopy invariant that
can be attached to continuous mappings (up to normalization).
\begin{theorem}[Existence and uniqueness of the degree]\label{exundeg}
There exists a unique integer-valued function
\[
(\opA,D,z)\mapsto \deg(\opA,D,z),
\]
where $D \subset \RRn$ is an open, bounded and nonempty subset,
 $\opA\in C(\ol{D},\RRn)$ and $z\not\in \opA(\partial D)$,
such that the following properties hold true:
\begin{enumerate}[label=\normalfont(\roman*)]
    \item (Normalization) If $z\in D$, then we have $\deg(\mathrm{id},D,z)=1$.
	\item (Additivity) Let $D_1,D_2\subset D$ be disjoint open subsets such that $z\not\in \opA(\ol{D}\setminus(D_1\cup D_2))$. Then
	$\deg(\opA,D,z)=\deg(\opA,D_1,z) + \deg(\opA,D_2,z)$.
    \item (Homotopy invariance) 
	If $\opK\in C(\ol{D}\times[0,1],\RRn)$ and $z\not\in\opK(\partial D\times [0,1])$, then 
	 $\lambda\mapsto\deg(\opK(\cdot, \lambda),D,z)$  is constant. 
\end{enumerate}
\end{theorem}

Among the many useful properties of the degree we highlight the following few.
\begin{corollary}\label{degcoro}
\leavevmode
\begin{enumerate}[label=\normalfont(\roman*)]
\item (Excision property)
Let $D \subset \RRn$ be an open, bounded and $D'\subset D$ open. If $z\not\in\opA(\ol{D}\setminus D')$, then $\deg(\opA,D,z)=\deg(\opA,D',z)$.
\item (Existence property)
If $z\not\in\opA(\ol{D})$, then $\deg(\opA,D,z)=0$. Said differently, if $\deg(\opA,D,z)\neq 0$ for some open, bounded $D\subset\RRn$
such that $z\not\in\opA(\partial D)$, then there is a solution $u\in D$ to $\opA(u)=z$.
\item (Additivity property) Suppose that $\{D_j\}\subset D$ is a sequence of open and disjoint sets. 
If $z\not\in\opA(\ol{D}\setminus\bigcup D_j)$, then $\deg(\opA,D_j,z)=0$ for all
but finitely many $j$, and 
$\deg(\opA,D,z)=\sum_j \deg(\opA,D_j,z)$.
\end{enumerate}
\end{corollary}

The following formula is essential for the practical computation of the degree. 
\begin{proposition}\label{degsum}
Let $\opA\in C(\ol{D},\RRn)\cap C^1(D,\RRn)$ such that $z\not\in\opA(\partial D)$ and $\det \opA'(u)\neq 0$ for all $u\in\opA^{-1}(z)$. Then
$\deg(\opA,D,z)=\sum_{u\in \opA^{-1}(z)} \sgn \det \opA'(u)$.
\end{proposition}

Based on this, we call a point $u$ \emph{non-degenerate} if $\det \opA'(u)\neq 0$  and \emph{degenerate} otherwise.
When talking about an \emph{isolated solution} $u$ of $\opA$, i.e. a point $u\in\opA^{-1}(z)$ such that there is a ball $B(u,r)$ with 
$\opA^{-1}(z)\cap B(u,r)=\{u\}$, the concept of the \emph{index} is useful. 
\begin{definition}
Let $\opA\in C(\ol{D},\RRn)$ and let $u$ be an isolated solution. Then the \emph{(topological) index} of $\opA$ at $u$ is defined as
$i(\opA,u)=\deg(\opA, B(u,r),z)$,
where $\opA^{-1}(z)\cap B(u,r)=\{u\}$. 
\end{definition}

It follows from the excision property that $i(\opA,u)$ is independent of $r$ (for sufficiently small values of $r>0$), so the definition makes sense.

\begin{proposition}\label{indxsgndet}
\leavevmode
\begin{enumerate}[label=(\roman*)]
\item Let $\opA\in C(\ol{D},\RRn)$ such that $z\not\in\opA(\partial D)$ and $\opA^{-1}(z)$ is finite. Then
$\deg(\opA,D,z)=\sum_{u\in \opA^{-1}(z)} i(\opA,u)$.
\item Let $u\in\RRn$ be a zero of $\opA\in C(\ol{D},\RRn)$,  where $D$ an open neighborhood of $u$. If $\opA$ is differentiable at $u$ with $\det \opA'(u)\neq 0$, then
$i(\opA,u)=\sgn\det \opA'(u)$.
\end{enumerate}
\end{proposition}

The topological degree (resp. index) can be naturally extended to continuous mappings of type $\opA : X\to Y$, where $X$ and $Y$ are $n$-dimensional oriented topological vector spaces
and the degree (resp. index) is independent of the choice of the basis. We refer the reader to \cite[Chapter 6]{dinca2009brouwer} for details.
\begin{theorem}\cite[Theorem 1.3.1]{dinca2021brouwer} (\cite[Theorem 6.3.1]{dinca2009brouwer})\label{topvecindx}
Let $\opA : D\to Y$ be continuous, where $X$ and $Y$ are $n$-dimensional topological vector spaces and $D\subset X$ is an open neighborhood of $u\in X$.
Let $h:X\to\RR^n$ and $g:Y\to\RR^n$ be linear homeomorphisms.
Suppose that $\opA$ is differentiable at $u$ and that $\opA'(u):X\to Y$ is invertible. Then $i(\opA,u)=\sgn\det g\opA'(u)h^{-1}$.
\end{theorem}

The following result says the degree is stable under (almost all) small perturbations of the right-hand side of $\opA(u)=z$, 
and that the degree provides a lower bound for the number of solutions
of the perturbed equation.

\begin{theorem}\cite[Corollary 7.4]{cronin1995fixed}\label{croninperturb}
Let $D\subset\RR^n$ be a bounded open set and let $\opA:D\to\RR^n$ be a $C^1$ mapping.
If $z\not\in\opA(\partial D)$, then there is a $\delta>0$ such that if $z'\in B(z,\delta)\setminus E_{z,\delta}$,
then 
$$
\deg(\opA,D,z)=\deg(\opA,D,z'),
$$
where $E_{z,\delta}\subset \{w\in B(z,\delta) : \det\opA'(u)=0,\; \opA(u)=w\}$.
Furthermore, $\opA^{-1}(z')$ consists of a finite number $m$ of points, where $|\deg(\opA,D,z)|\le m$ and $m\equiv\deg(\opA,D,z)$ (mod 2).
\end{theorem}
The proof is based on Sard's theorem, which says that $E_{z,\delta}$ has $n$-dimensional Lebesgue measure zero.
Next, we recall a tool that is very useful for the computation of the degree in the degenerate case.
\begin{theorem}\cite[Theorem 6.4.2]{dinca2009brouwer}\label{leraysec}
Let $X$ and $Z$ be $n$-dimensional topological vector spaces.
Let $L:X\to Z$ be a linear mapping with $\ker L\neq\{0\}$ and $V\subset Z$ be a vector space such that $Z=V\oplus\Ran L$.
Let $D\subset X$ open and bounded, $r:\ol{D}\to V$ continuous with $0\not\in \opA(\partial D)$, where $\opA=L+r$. Then, for each invertible linear map $J : \ker L\to V$,
and each projector $P$ with $\Ran P=\ker L$, the relation
$$
\deg(\opA,D,0)=i(L+JP,0) \deg(J^{-1}r|_{\ker L}, D\cap\ker L, 0)
$$
called \emph{Leray's second reduction formula} holds true.
\end{theorem}

\begin{corollary}\label{noninvpert}
Let $L:X\to Z$ be a linear mapping with $\ker L\neq\{0\}$ and let $Q$ be a projector with $\ker Q=\Ran L$.
Also, let $\mathcal{N}:D\times[0,1]\to Z$ be continuous, with $D\subset X$ open and bounded.
Furthermore, assume the following hold true.
\begin{enumerate}[label=\normalfont(\roman*)]
\item $Lu+\lambda\mathcal{N}(u,\lambda)\neq 0$ for all $u\in\partial D$ and $\lambda\in(0,1]$.
\item $Q\mathcal{N}(u,0)\neq 0$ for each $u\in \partial D$.
\end{enumerate}
Then 
$$
\deg(L+\mathcal{N}(\cdot,1), D, 0)=i(L+Q)\deg(Q\mathcal{N}(\cdot,0)|_{\ker L}, D\cap\ker L, 0).
$$
\end{corollary}
\begin{proof}
We adapt the the proof of \cite[Theorem 2.2.3]{dinca2021brouwer}.
Define the homotopy
$$
\opA(u,\lambda)=Lu + (1-\lambda) Q \mathcal{N}(u,\lambda) + \lambda \mathcal{N}(u,\lambda).
$$
Fix $\lambda\in(0,1]$. Projecting the equation $\opA(u,\lambda)=0$ with $Q$ and $I-Q$ (i.e. onto the complementary spaces $\Ran Q$ and $\Ran(I-Q)=\ker Q=\Ran L$),
we get that $\opA(u,\lambda)=0$ is equivalent to the system $Q\mathcal{N}(u,\lambda)=0$, $Lu + \lambda (I-Q)\mathcal{N}(u,\lambda)=0$.
But this is equivalent to $Lu + \lambda \mathcal{N}(u,\lambda)=0$. By assumption (i), $\opA(u,\lambda)\neq 0$ for all $u\in\partial D$ and $\lambda\in(0,1]$.
Further, $\opA(u,0)=0$ is equivalent to the system $u\in\ker L$, $Q\mathcal{N}(u,0)=0$, hence by assumption (ii), $\opA(u,0)\neq 0$ for all $u\in\partial D$.
Using the homotopy invariance and Leray's second reduction formula with $V=\ker L$, $J=I$ and $r=\mathcal{N}(\cdot,0)$, we get
$$
\deg(\opA(\cdot,1),D,0)=\deg(\opA(\cdot,0),D,0)=i(L+Q,0)\deg(Q\mathcal{N}(\cdot,0)|_{\ker L}, D\cap\ker L, 0).\qedhere
$$
\end{proof}

\subsection{Orientation-preserving mappings}
An important class of mappings for which the degree behaves rather nicely is the following.

\begin{definition}
Let $U\subset\RRn$ be open.
\begin{enumerate}[label=\normalfont(\roman*)]
\item A linear map $L:\RRn\to\RRn$ is said to be \emph{orientation-preserving} if $\det L\ge 0$.
\item A mapping $\opA\in C^2(U,\RRn)$ is said to be \emph{strictly orientation-preserving in $U$} if 
\begin{enumerate}[label=(\alph*)]
	\item $\opA'(u)$ is orientation-preserving for all $u\in U$, and
	\item the set $\{u\in U : \det\opA'(u)=0\}$ is nowhere dense in $U$.
\end{enumerate}
\item A mapping $\opA\in C(U,\RRn)$ is said to be \emph{orientation-preserving} if for every $V\subset U$ open with $\ol{V}\subset U$ is compact, there exists 
a sequence $\{\opA_j\}$ such that $\opA_j$ is strictly orientation preserving for all $j$ and $\opA_j\to\opA$ uniformly on $V$.
\end{enumerate}
\end{definition}

\begin{theorem}\label{opthm} \cite[Theorem 7.2.1]{dinca2021brouwer}  (\cite[Theorem 19.3.1]{dinca2009brouwer})
Let $U\subset\RRn$ be open, $D\subset U$ an open bounded set with $\ol{D}\subset U$ and let $\opA:U\to\RRn$ be orientation-preserving.
If $z\not\in\opA(\partial D)$, then the following hold true. 
\begin{enumerate}[label=\normalfont(\roman*)]
\item $\deg(\opA,D,z)\ge 0$.
\item $\deg(\opA,D,z)>0$ if and only if $z\in \opA(D)$.
\item If $\deg(\opA,D,z)=1$, then $\opA^{-1}(z)\cap D$ is connected. 
\item If $\opA(D)$ contains a point of some component $C$ of $\RRn\setminus\opA(\partial D)$, then $C\subset\opA(D)$.
\end{enumerate}
\end{theorem}

The preceding notions are related to the well-known monotone-type mappings.
\begin{theorem}\label{monothm}\cite[Proposition 7.2.1]{dinca2021brouwer} (\cite[Proposition 19.4.1]{dinca2009brouwer})
Let $U\subset\RRn$ be open, and suppose that $\opA:U\to\RRn$ is monotone, i.e. 
$$
\dua{\opA(u)-\opA(v)}{u-v}_{\RR^n}\ge 0
$$
for all $u,v\in U$. Then $\opA$ is orientation-preserving. 
\end{theorem}

\subsection{Holomorphic mappings}
Next, we consider the complex case.
\begin{definition}
Let $U\subset\CCn$ be open. A complex mapping $\opA:U\to\CCn$ is said to be \emph{holomorphic at $a\in\CCn$} if there is a $\CC$-linear mapping $L_a:\CCn\to\CCn$ and
a mapping $r_a: U\setminus\{a\}\to\CCn$ with $r_a=o(1)$, such that
$$
\opA(a+h)=\opA(a)+L_ah + r_a(h)
$$
for all $h\in\CC$ such that $a+h\in U$. In this case, $L_a=\opA'(a)$, where $\opA'(a)h=h_1 \partial_{z_1} \opA(a) + \ldots + h_n \partial_{z_n} \opA(a)$
and $\partial_{z_k} \opA(a)$ denotes the $k^\text{th}$ complex partial derivative of $\opA$ at $a$.
\end{definition}

We now remind the reader of some elementary linear algebra \cite{prasolov1994problems}.
 Every complex vector space $V$ is also vector space over $\RR$. This space will be denoted as $V_\RR$ and called the
\emph{realification} of $V$. The realification of a linear operator $A:V\to W$ over $\CC$ is the linear map $A_\RR : V_\RR\to W_\RR$.
If $\{e_1,\ldots,e_n\}$ is a basis in $V$, then $\{e_1,\ldots,e_n,ie_1,\ldots,ie_n\}$ is a basis in $V_\RR$. Further, if $\{e_1',\ldots,e_m'\}$ is a basis in $W$,
and $A=B+iC$ with some real matrices $B$ and $C$, then the matrix of $A_\RR$ with respect to the bases $\{e_1,\ldots,e_n,ie_1,\ldots,ie_n\}$ and
$\{e_1',\ldots,e_n',ie_1',\ldots,ie_n'\}$ is 
$$
\begin{pmatrix}
B & -C \\
C & B
\end{pmatrix}.
$$
The determinant of the realification $A_\RR$ obeys the following important rule: 
\begin{equation}\label{cdet}
\det A_\RR=|\det A|^2.
\end{equation}

This motivates that we define the realification of a mapping $\opA:U\to\CCn$ with $\opA=(\opA_1,\ldots,\opA_n)$ as $\opA_\RR : U\to\RR^{2n}$ 
 via
$$
\opA_\RR(x,y)=(\Real\opA_1(z),\Imag\opA_1(z),\ldots,\Real\opA_n(z),\Imag\opA_n(z)),
$$
where $x=(x_1,\ldots,x_n)$, $y=(y_1,\ldots,y_n)$ and $z=x+iy$. Notice that we do not explicitly denote the realification of the set $U$.
\begin{definition}
Let $U\subset\CCn$ be open, $\opA:U\to\CCn$ be a holomorphic mapping and $D\subset U$ a domain. The \emph{degree of $\opA$ in $D$} is defined as the degree of the realification,
$\deg(\opA,D,z):=\deg(\opA_\RR,D,z)$.
\end{definition}

It can be proved using \cref{cdet}, that if $\opA:U\to\CCn$ is holomorphic, then
$\det \opA_\RR'(x,y)=|\!\det \opA'(z)|^2$
for all $z\in U$. This innocent-looking identity has striking consequences.
\begin{theorem}\label{cmplxdeg}\cite[Section 19.5]{dinca2009brouwer}
Let $U\subset\CCn$ be open, $\opA:U\to\CCn$ be a holomorphic mapping and $D\subset U$ bounded and open. Then the following statements hold true.
\begin{enumerate}[label=\normalfont(\roman*)]
\item $\opA_\RR : U\to\RR^{2n}$ is orientation-preserving. In particular, $\deg(\opA,D,z)\ge 0$ and $\deg(\opA,D,z)>0$ if and only if $z\in \opA(D)$.
\item Let $\ol{D}\subset U$ and $z\not\in\opA(\partial D)$. If $\deg(\opA,D,z)=k$, then $\opA(u)=z$ has at most $k$ solutions in $D$.
 If $\deg(\opA,D,z)=1$, then $\opA(u)=z$ has a unique solution in $D$.
\item Let $\zeta\in D$ be a solution of $\opA(\zeta)=z$ and denote by $C$ the connected component of $\opA^{-1}(z)$ containing $\zeta$. Then \emph{either} $\zeta$ 
is an isolated solution of $\opA(\zeta)=z$ (and hence $C=\{\zeta\}$) \emph{or} $C\cap G\neq\emptyset$ for any neighborhood $G$ of $\partial D$. 
\item Let $\ol{D}\subset U$ and $z\not\in\opA(\partial D)$. Then $\deg(\opA,D,z)=1$ if and only if there is a unique non-degenerate solution $\zeta\in D$ of $\opA(\zeta)=z$.
\item \cite[Theorem 42 (b)]{zabrejko1997rotation} The number of zeros in $D$ is finite.
\end{enumerate}
\end{theorem}

From (iv) we can conclude that if $\zeta$ is a degenerate isolated solution of $\opA(\zeta)=z$, then necessarily $i(\opA,\zeta)\ge 2$.
Further in the holomorphic case, \cref{croninperturb} can sharpened as follows.
\begin{theorem}\label{holocronin}
Let $\opA:D\to\CCn$ be a holomorphic mapping and $D\subset\CCn$ bounded and open.
If $z\not\in\opA(\partial D)$, then there is a $\delta>0$ such that if $z'\in B(z,\delta)\setminus E$,
then 
$$
\deg(\opA,D,z)=\deg(\opA,D,z'),
$$
where $E\subset \{w\in B(z,\delta) : \det\opA'(u)=0,\; \opA(u)=w\}$ has measure 0.
Furthermore, $\opA^{-1}(z')$ consists of a finite number $m$ of points, where $m=\deg(\opA,D,z)$.
\end{theorem}
\begin{proof}
From \cref{croninperturb}, we have $\deg(\opA,D,z)\le m$ and the converse
 inequality $m\le \deg(\opA,D,z')=\deg(\opA,D,z)$ follows from \cref{cmplxdeg} (ii).
\end{proof}

The following result is concerned with holomorphic extensions.
\begin{theorem}\cite[Theorem 45.]{zabrejko1997rotation}\label{holodegrel}
Let $\opA:\RRn\to\RRn$ be a real analytic mapping that can be extended to a holomorphic mapping $\wt{\opA}:\CCn\to\CCn$.
Let $D\subset\RRn$ be an open bounded set such that $0\not\in \opA(\partial D)$. Let 
$$
D_\epsilon=\{ x\in \CCn : \Real x\in D,\quad |\Imag x|<\epsilon \}.
$$
Then, for sufficiently small $\epsilon>0$, 
$$
|\deg(\opA,D,0)|\le \deg(\wt{\opA}, D_\epsilon, 0), \quad \deg(\opA,D,0)=\deg(\wt{\opA},D_\epsilon,0) \mod 2.
$$
\end{theorem}

\section{Analysis of the SRCC method}\label{secsrcc}

As stressed earlier, we consider the finite-dimensional case only, by which we mean that the cardinality of the orbital basis $\{\phi_p\}_{p=1}^K$, $K$ is finite.
Recall the definition of $\HC^1_K$ from \cref{hamsec} and the interpretation of the projected Hamiltonian from \cref{truncsec}.

\subsection{Definitions and basic properties}\label{ccbasic}

Let $\VV:=\VV(G)$ be the real amplitude space corresponding to some consistent excitation graph $G$ (see Section 3.5 of Part I). 
In particular, one can take $\VV$ to be the S, D, SD, SDT, etc. truncated amplitude space, or the full amplitude space, see \cref{clusampsec}.
Let $\VC$ be the corresponding functional amplitude space, i.e. the set of wavefunctions of the form $T\Phi_0$ for any $t\in\VV$, where $T$ is
the cluster operator corresponding to $t$---this convention will be used throughout. 
 
Define the mapping $\opA:\VV\to \VV^*$ via the instruction
\begin{equation}\label{FCC}
\vdua{\opA(t)}{s}:=\hdua{e^{-T}\ham_K e^T\Phi_0}{S\Phi_0},
\end{equation}
for any $t,s\in \VV$.
Then, $\opA$ is well-defined, because \cref{FCC} can be rewritten as $\dua{\opA(t)}{s}=\dua{\ham_K e^T\Phi_0}{e^{-T^\dag}S\Phi_0}=\dua{\ham e^T\Phi_0}{e^{-T^\dag}S\Phi_0}$,
and here $e^T\Phi_0\in\HC^1_K$ and $e^{-T^\dag}S\Phi_0\in\HC^1_K$.

Recall the definition (see (2.6) or Theorem 4.4 of Part I) of the CC energy,
$$
\ene_\cc(t):=\hdua{e^{-T}\ham_K e^T\Phi_0}{\Phi_0}=\hdua{\ham_K e^T\Phi_0}{\Phi_0},
$$
where the second equality follows from $(e^{-T})^\dag\Phi_0=\Phi_0$.
Note that $\ene_\cc(t)\in\RR$ since the amplitude space $\VV$ is assumed to be real.
The similarity-transformed Hamiltonian occurs often in the forthcoming discussion, so we introduce the notation
\begin{equation}\label{simham}
\ham_K(t):=e^{-T}\ham_K e^T:\HC_K^1\to(\HC_K^1)^*,
\end{equation}
which is a bounded map for any fixed cluster amplitude $t\in\VV$. Furthermore, for a given bounded map $\mathcal{T}:\HC_K^1\to(\HC_K^1)^*$, we define
the operator $\mathcal{T}_\VC : \VC\to\VC$ via $\dua{\mathcal{T}_\VC \Psi}{\Psi'}=\hdua{\mathcal{T}\Psi}{\Psi'}$ for all $\Psi,\Psi'\in\VC$.
For instance, we will often encounter the projected similarity-transformed Hamiltonian $\ham_K(t)_\VC$.

We will also use a notation analogous to \cref{simham} for the similarity-transformed fluctuation operator $\mathcal{W}_K$ (see \cref{fockfluc}), i.e. 
$\mathcal{W}_K(t) = e^{-T}\mathcal{W}_K e^T$.
In our finite-dimensional setting, the similarity-transformed Fock operator can be given explicitly as
\begin{equation}\label{fockcomm}
e^{-T}\mathcal{F}_Ke^T=\mathcal{F}_K + [\mathcal{F}_K,T],\quad\text{and}\quad [\mathcal{F}_K,X_\alpha]=\epsilon_\alpha X_\alpha,
\end{equation}
for any $t\in\VV$, see e.g. \cite[Lemma 15]{faulstich2019analysis}. In particular,
\begin{equation}\label{simcomm}
[[\ham_K(t),U],V]=[[\mathcal{W}_K(t),U],V],
\end{equation}
for any $t,u,v\in\VV$.

The following simple observation shows the equivalence of the (strong) Schr\"odinger equation with the Full CC method.
\begin{lemma}\label{fcclemm}
Assume that the Slater determinant basis satisfies $\Phi_\alpha\in\HC_K^2$.
Suppose that $\VV=\VV(G^\full)$, and 
 that $\opA(t_*)=0$. Then the function $\Psi=(c_0 I + C)\Phi_0\in\HC_K^2$ satisfies the Schr\"odinger equation $\ham_K\Psi=\ene\Psi$ if and only if
$e^{-T_*}(c_0 I + C)=r_0 I + R$, where ($c_0 = r_0$ and)
\begin{equation}\label{fcclemmsys}
\left.
\begin{aligned}
\ene_\cc(t_*)r_0 + \dua{\ham_K(t_*)R\Phi_0}{\Phi_0} &= \ene r_0 \\
\ham_K(t_*)_\VC R\Phi_0 &= \ene R\Phi_0
\end{aligned}
\right\}
\end{equation}
Furthermore,
\begin{equation}\label{hamspec}
\sigma(\ham_K)=\{\ene_{\cc}(t_*)\}\cup \sigma(\ham_K(t_*)_\VC).
\end{equation}
\end{lemma}
\begin{proof}
Using the splitting $\Span\{\Phi_0\}\oplus \VC$, the similarity-transformed Hamiltonian is block upper triangular in the Slater basis,
$$
\ham_K(t_*) = \begin{pmatrix}
\ene_\cc(t_*) & \dua{\ham_K(t_*)\,\cdot\,}{\Phi_0} \\
0 & \ham_K(t_*)_\VC
\end{pmatrix},
$$
due to $\dua{\ham_K(t_*)\Phi_0}{\Phi_\alpha}=0$. The proof now follows by noting that the eigenvalues of $\ham_K(t_*)$ and $\ham_K$ are the same,
and the eigenvectors of $\ham_K(t_*)$ are of the form $e^{-T_*}\Phi$, where $\ham_K\Phi=\ene'\Phi$.  Formula \cref{hamspec} follows by the fact that the spectrum 
of a block triangular matrix is the union of the spectra of the blocks in the diagonal.
\end{proof}

We distinguish two cases. Obviously, $r_0\neq 0$ and $R=0$ is a solution to the system \cref{fcclemmsys} if and only if $\ene_\cc(t_*)=\ene$.
In this case, 
$\Psi=e^{T_*}\Phi_0$ is a solution.

If, however, $\ene_\cc(t_*)\neq \ene$, then $R$ cannot be 0 (because that would imply $\Psi=0$). In this case, $\Psi=(r_0 I + R)e^{T_*}\Phi_0$, where 
$
r_0=\dua{\ham_K(t_*)R\Phi_0}{\Phi_0}/(\ene-\ene_\cc(t_*)).
$
Note that it is possible to have $r_0=0$, in which case $\dua{\Psi}{\Phi_0}=0$.
We return to this latter case in \cref{srcceom} and discuss the case $\ene_\cc(t_*)=\ene$ below.
\begin{remark}\label{fccdegr}
If $\ene_\cc(t_*)=\ene$, then \cref{fcclemmsys} reduces to 
$$
\left.
\begin{aligned}
\dua{\ham_K(t_*)R\Phi_0}{\Phi_0} &= 0\\
\ham_K(t_*)_\VC R\Phi_0 &= \ene R\Phi_0
\end{aligned}
\right\}
$$
Suppose that this system has $\mu$ linearly independent solutions $R_1,\ldots,R_\mu$. Then it is easy to see, 
using $\Psi=(r_0 I + R_\mu)e^{T_*}\Phi_0$, that the wavefunctions 
$$
\{ e^{T_*}\Phi_0, R_1e^{T_*}\Phi_0, \ldots, R_\mu e^{T_*}\Phi_0 \}
$$
span the eigenspace $\ker(\ham_K-\ene)$. In particular, we have $\dim\ker(\ham_K-\ene)=\mu+1$.
Note also that in this case $\sigma(\ham_K)=\sigma(\ham_K(t_*)_\VC)$.
\end{remark}

Let us now recall that the CC equation $\opA(t_*)=0$ can be cast in a form that closely resembles the CI eigenvalue problem (see (2.3) of Part I) (although it is \emph{not} 
equivalent to it in general).
\begin{lemma}\cite[Theorem 5.6]{schneider2009analysis}\label{unlinkedcc}
Let $G$ be excitation complete, and let $\VV=\VV(G)$ be the corresponding amplitude space. 
Then the ``linked'' CC equation $\opA(t_*)=0$ is equivalent to the ``unlinked'' (a.k.a. ``energy-dependent'') CC equation
\begin{equation}\label{ccunlinked}
\hdua{\ham_K e^{T_*}\Phi_0}{S\Phi_0}=\ene_\cc(t_*) \hdua{e^{T_*}\Phi_0}{S\Phi_0}\quad\text{for all}\quad s\in\VV.
\end{equation}
\end{lemma}
\begin{proof}
We have
\begin{align*}
&\hdua{(\ham_K-\ene_\cc(t_*))e^{T_*}\Phi_0}{S\Phi_0}=\hdua{e^{-T_*}(\ham_K-\ene_\cc(t_*)) e^{T_*}\Phi_0}{(e^{T_*})^\dag S\Phi_0}\\
&=\hdua{e^{-T_*}(\ham_K-\ene_\cc(t_*)) e^{T_*}\Phi_0}{\Pi_\VC(e^{T_*})^\dag S\Phi_0} + \hdua{e^{-T_*}(\ham_K-\ene_\cc(t_*)) e^{T_*}\Phi_0}{\underbrace{\Pi_{\Phi_0}(e^{T_*})^\dag S\Phi_0}_{\mathrm{const}\cdot\Phi_0}}\\
&= \hdua{e^{-T_*}\ham_K e^{T_*}\Phi_0}{\Pi_\VC(e^{T_*})^\dag S\Phi_0},
\end{align*}
where second term on the right-hand side of the penultimate equality vanishes by the definition of $\ene_\cc(t_*)$.
The proof is completed by recalling that $\Pi_\VC(e^{T_*})^\dag:\VC\to\VC$ is surjective due to Proposition 3.30 of Part I.
\end{proof}

The ``unlinked'' form is less useful in practice, because the expansion of $\ham_K e^T$ does not terminate like the Baker--Campbell--Hausdorff series
\begin{equation}\label{bchagain}
\ham_K(t)=\sum_{j=0}^4 \frac{1}{j!} [\ham_K,T]_{(j)}.
\end{equation}
More generally, the \emph{doubly}\footnote{The doubly similarity-transformation above differs from the one considered in Arponen's Extended CC (ECC) theory \cite{Arponen1983}.} similarity-transformed Hamiltonian $\ham_K(t+s)=e^{-S}\ham_K(t)e^S$ can also be expanded using the Baker--Campbell--Hausdorff series
but in this case 
\begin{equation}\label{bchdoub}
\ham_K(t+s)=\sum_{j=0}^{2N} \frac{1}{j!} [\ham_K(t),S]_{(j)},
\end{equation}
i.e. the series terminates at $2N$. To see this, simply note that $[\ham_K(t),S]_{(2N+1)}$ consists of terms of the form $S^i \ham_K(t) S^k$, where $i+k=2N+1$,
so $i,k\ge N+1$, which, using the nilpotency of the cluster operators implies that all terms for $j\ge 2N+1$ vanish.

\subsection{Local properties---real case}\label{secsrccop}

Next, we look at the local behavior of the CC mapping $\opA:\VV\to\VV^*$ for general (real) amplitude spaces $\VV$.
For fixed $t\in\VV$, define the
\emph{modified similarity-transformed Hamiltonian},
\begin{equation}\label{hamhats}
\wh{\ham_K}(t):= \ham_K(t) - \sum_{\alpha\in\Xi(G)^c}\hdua{\ham_K(t)\Phi_0}{\Phi_\alpha}X_\alpha,
\end{equation}
where $\Xi(G)^c=\Xi(G^\full)\setminus \Xi(G)$ and the \emph{set of excitations} $\Xi(G)$ was defined in (3.3) of Part I.
For example, if SD truncation is in effect, then $\Xi(G)$ consists of all singly-, and doubly excited multiindices (see \cref{exopsec}).
Also, if $\VV$ is the full amplitude space, then $\Xi(G)$ consists of all multiindices $\alpha\neq 0$.

\begin{definition}\label{rankreg}
The amplitude space $\VV(G)$ is said to be \emph{rank-regular}, if 
$$
\dua{X_\alpha \Phi_\beta}{\Phi_\gamma}=0\quad\text{for all $\beta,\gamma\in\Xi(G)$ and $\alpha\in\Xi(G)^c$.}
$$
\end{definition}
We immediately get that $\wh{\ham_K}(t)_\VC=\ham_K(t)_\VC$ if $\VV(G)$ is rank-regular. 
The next proposition shows that the truncated subgraphs typically used in practice---such as ones coming from S, D, SD, SDT, etc. truncations---are rank-regular. We refer the reader to Section 3.2 of Part I for details on truncated subgraphs of the excitation graph. 

\begin{proposition}\label{rkmodham}
Suppose that the excitation graph $G$ is a rank-truncated subgraph of the form $G=G(1,2,\ldots,\rho)$, for some $\rho=1,\ldots,N$ or $G=G(\mathrm{D})$.
Then $\VV(G)$ is rank-regular.
\end{proposition}
\begin{proof}
The set $\Xi(G)^c$ consists of elements of rank $\rho+1,\ldots,N$ (or empty), so that
$\dua{X_\alpha \Phi_\beta}{\Phi_\gamma}=0$ for all $\rk\alpha\not\in\{1,2,\ldots,\rho\}$ and all $\rk\beta,\rk\gamma\in\{1,2,\ldots,\rho\}$,
due to the fact that $X_\alpha \Phi_\beta$ is of rank $\rk(\alpha)+\rk(\beta)\not\in\{1,2,\ldots,\rho\}$.
The proof of the case $G=G(\mathrm{D})$ is similar.
\end{proof}

Next, we compute the derivative of the mapping $\opA$, which explains the definition of $\wh{\ham_K}(t)$ and of rank-regularity.

\begin{lemma}\label{Adifflemm}
Let $t_*$ be a zero of $\opA:\VV\to\VV^*$.
Then the derivative $\opA'(t_*)\in\mathcal{L}(\VV,\VV^*)$ is given by
\begin{equation}\label{SRCCDZ}
\vdua{\opA'(t_*)u}{v}=\hdua{(\widehat{\ham_K}(t_*) - \ene_{\cc}(t_*))U\Phi_0}{V\Phi_0}
\end{equation}
for all $u,v\in\VV$.
\end{lemma}
\begin{proof}
The derivative $\opA':\VV\to \mathcal L(\VV,\VV^*)$ is readily computed as
\begin{align*}
\left.\frac{d}{dh}\vdua{\opA(t+hu)}{v}\right|_{h=0}&=\left.\frac{d}{dh} \hdua{e^{-T-hU}\ham_K e^{T+hU}\Phi_0}{V\Phi_0}\right|_{h=0}\\
&=\left.\hdua{e^{-T-hU}(\ham_K U - U \ham_K)e^{T+hU}\Phi_0}{V\Phi_0}\right|_{h=0}\\
&=\hdua{e^{-T}(\ham_K U - U \ham_K)e^{T}\Phi_0}{V\Phi_0},
\end{align*}
so using the commutativity of the cluster operators, we get 
\begin{equation}\label{SRCCD}
\vdua{\opA'(t)u}{v}=\hdua{[\ham_K(t),U]\Phi_0}{V\Phi_0}
\end{equation}
for all $t,u,v\in\VV$.
Expanding $U^\dag V\Phi_0\in\HC_K^1$ in the $\LC^2$-orthonormal basis $\{\Phi_\alpha\}_{\alpha}\subset\HC^1_K$, 
$$
U^\dag V\Phi_0=\sum_{\alpha} \dua{U\Phi_\alpha}{V\Phi_0} \Phi_\alpha,
$$
we obtain using $\hdua{\ham_K(t_*)\Phi_0}{\Phi_\alpha}=0$ for all $\alpha\in\Xi(G)$, 
\begin{align*}
\hdua{U \ham_K(t_*)\Phi_0}{V\Phi_0}&=\hdua{\ham_K(t_*)\Phi_0}{U^\dag V\Phi_0}\\
&= \ene_{\cc}(t_*)\dua{U\Phi_0}{V\Phi_0} + \sum_{\alpha\in\Xi(G)^c}\hdua{\ham_K(t_*)\Phi_0}{\Phi_\alpha}\dua{X_\alpha U\Phi_0}{V\Phi_0}
\end{align*}
for all $u,v\in\VV$. 
Inserting this into \cref{SRCCD} with $t=t_*$, we obtain the stated formula.
\end{proof}

As we noted in \cref{secprev}, previous analyses of the CC mapping assumed the local strong monotonicity at a zero $t_*$, i.e. that there is a $\delta>0$ and a
 constant $C_{\mathrm{SM}}(t_*,\delta)>0$ such that
\begin{equation}\label{Asmco}
\vdua{\opA(t)-\opA(s)}{t-s}\ge C_{\mathrm{SM}}(t_*,\delta) \|t-s\|_\VV^2, \quad\text{for all}\quad t,s\in B_\VV(t_*,\delta).
\end{equation}
The following elementary theorem makes the observations in \cite{rohwedder2013error} more precise.
\begin{theorem}\label{thm:SM-CC}
Let $t_*\in\VV$ be a zero of $\opA:\VV\to\VV^*$.
\begin{enumerate}[label=\normalfont(\roman*)]
\item If $\opA$ is strongly monotone in $B_\VV(t_*,\delta)$ for some $\delta>0$,
then there exists $\delta'>0$ such that $\opA'(t_*+u)$ is $\VV$-coercive for all $\|u\|_\VV<\delta'$ with some constant $0<\gamma\le C_{\mathrm{SM}}(t_*,\delta)$, i.e.
\begin{equation}\label{apcoerc}
\vdua{\opA'(t_*+u)v}{v}\ge \gamma \|v\|_\VV^2 \quad \text{for all}\quad v\in\VV\;\text{and}\;\|u\|_\VV<\delta'.
\end{equation}
\item Conversely, if \cref{apcoerc} holds with $u=0$, then \cref{Asmco} holds true with $\delta>0$ chosen so that $C_{\mathrm{SM}}(t_*,\delta):=\gamma- M_\delta \delta>0$, where 
\begin{equation}\label{mdelta}
M_\delta=\sup_{\|\zeta\|_\VV\le\delta} \|\opA''(t_*+\zeta)\|_{\mathcal{L}(\VV\times\VV,\VV^*)}.
\end{equation}
\end{enumerate}
\end{theorem}
\begin{proof}
To see (i), fix $\delta'>0$ and $\|u\|_\VV<\delta'$ and write for any $\|r\|_\VV<\delta'<\delta$, 
\begin{align*}
    C_{\mathrm{SM}}(t_*,\delta)\|r-u\|_\VV^2 &\le \vdua{\opA(t_*+r)-\opA(t_*+u)}{r-u}\le \vdua{\opA'(t_*+u)(r-u)}{r-u}  + \frac{1}{2} M_{\delta} \|r-u\|_\VV^3.
\end{align*}
This implies
$$
(C_{\mathrm{SM}}(t_*,\delta)-M_{\delta}\delta')\|r-u\|_\VV^2 \le \vdua{\opA'(t_*+u)(r-u)}{r-u}.
$$
Any vector $v\in\VV$ can be expressed as $v=\alpha(r-u)$ for some $\alpha>0$ and $\|r\|_\VV<\delta'$,
from which $\VV$-coercivity follows with $\gamma=C_{\mathrm{SM}}(t_*,\delta)-M_\delta \delta'$, by choosing $\delta'$ sufficiently small.

Next, to prove (ii), write the Taylor expansions of $\opA$ at $t_*$,
\begin{align*}
\opA(t_*+r)&=\opA'(t_*)r + \mathcal{R}_2(t_*;r),\\
\opA(t_*+r')&=\opA'(t_*)r' + \mathcal{R}_2(t_*;r'),
\end{align*}
for any $\|r\|_\VV ,\|r' \|_\VV<\delta$ for some $\delta>0$, from which we obtain
\begin{align*}
\vdua{\opA(t_*+r)-\opA(t_*+r')}{r-r'}&=\vdua{\opA'(t_*)(r-r')}{r-r'} + \vdua{\mathcal{R}_2(t_*;r)-\mathcal{R}_2(t_*;r')}{r-r'}\\
&\ge \gamma \|r-r'\|_\VV^2  + \vdua{\mathcal{R}_2(t_*;r)-\mathcal{R}_2(t_*;r')}{r-r'}.
\end{align*}
Using the intermediate value inequality, we have
\begin{align*}
\|\mathcal{R}_2(t_*;r)-\mathcal{R}_2(t_*;r')\|_{\VV^*}&\le M_{r,r'}\|r-r'\|_{\VV},
\end{align*}
where
\begin{align*}
M_{r,r'}&=\max_{\xi\in[r,r']} \|\partial_2 \mathcal{R}_2(t_*;\xi)\|=\max_{\xi\in[r,r']} \|\opA'(t_*+\xi) - \opA'(t_*)\|\\
&\le \Big(\max_{\xi\in[r,r']}\max_{\zeta\in[0,\xi]} \|\opA''(t_*+\zeta)\|_{\mathcal{L}(\VV\times\VV,\VV^*)}\Big) \delta \\
&= \Big(\sup_{\|\zeta\|\le\delta} \|\opA''(t_*+\zeta)\|_{\mathcal{L}(\VV\times\VV,\VV^*)}\Big)\delta = M_\delta \delta
\end{align*}
This implies $\vdua{\opA(t_*+r)-\opA(t_*+r')}{r-r'}\ge (\gamma -M_\delta \delta)\|r-r'\|_\VV^2$ for all $\|r\|_\VV ,\|r' \|_\VV<\delta$. % sufficiently close to 0. 
Setting $t=t_*+r$ and $s=t_*+r'$ proves the claim.
\end{proof}

\begin{remark}\label{smproj}
Let $\VV^0\subset\VV$ be a subspace and consider the \emph{projected CC mapping} $\opA^0 : \VV^0\to (\VV^0)^*$ via
$$
\dua{\opA^0(t^0)}{s^0}_{\VV^*\times\VV}=\dua{\opA(t^0)}{s^0}_{\VV^*\times\VV}\quad \text{for all}\quad t^0,s^0\in\VV^0.
$$ 
Clearly, if $\opA$ is strongly monotone on $B_{\VV}(t_*,\delta)$ with a constant $C_{\mathrm{SM}}>0$ at a zero $t_*$, then
$\opA^0$ is strongly monotone on $B_{\VV^0}(t_*^0,\sqrt{\delta^2-\|t_*^\perp\|_\VV^2})$ provided $\|t_*^\perp\|_\VV^2$ is small enough, with the same constant $C_{\mathrm{SM}}$, where we have set $t_*^0=\Pi_{\VV^0}t_*$
and $t_*^\perp=(I-\Pi_{\VV^0})t_*$. Here, $\Pi_{\VV^0}:\VV\to\VV$ denotes the $\ell^2$-orthogonal projector onto $\VV^0$.
Note that $t_*^0$ is \emph{not}, in general, a zero of $\opA^0$, hence the preceding theorem is not applicable to $\opA^0$.
\end{remark}

\begin{remark}\label{Mdeltarmk}
The quantity $M_\delta$ contains the second derivative of $\opA$. It is easy to see that
$$
\vdua{\opA''(t)(u,v)}{w}=\hdua{[[\ham_K(t),U],V]\Phi_0}{W\Phi_0}
$$
for all $u,v,w\in\VV$. Using \cref{simcomm}, we have
$\vdua{\opA''(t)(u,v)}{w}=\hdua{[[\mathcal{W}_K(t),U],V]\Phi_0}{W\Phi_0}$,
so that $\opA''(t)$ only involves the fluctuation operator $\mathcal{W}_K$.
\end{remark}

\begin{remark}[Perturbative regime]\label{pertreg}
Let $\VV(G)$ be rank-regular, and consider the case when $t_*\approx 0$, which is the case considered in \cite{schneider2009analysis,rohwedder2013error}.
Then roughly speaking, we have $\ham_K(t_*)\approx \ham_K$.
Note that
\begin{equation*}
    \vdua{\opA'(t_*)r}{r}=\hdua{(\ham_K - \ene_\cc(t_*))R\Phi_0}{R\Phi_0}  + \mathcal{O}(\|t_*\|_\VV) .
\end{equation*}
Consequently, if 
$$
\hdua{(\ham_K - \ene_\cc(t_*))R\Phi_0}{R\Phi_0}\geq c(t_*) \|r\|_\VV^2,
$$
where $c(t_*)>0$, then local strong monotonicity holds with  constant $C_\mathrm{SM} = c(t_*) -M'\Vert t_*\Vert_\VV - 2M_\delta\delta$ for $t_*$ sufficiently close to 0. 
In \cite[Lemma 3.5]{rohwedder2013error}, it is shown that such a $c(t_*)$ exists under the assumption that $\ham_K$ has a spectral gap and that $\Phi_0$ is a sufficiently good approximation of the ground state $e^{T_*} \Phi_0$ (i.e. that $t_*$ is sufficiently close to 0).
\end{remark}

\begin{proposition}\label{smnondeg}
If $\opA$ is locally strongly monotone at a zero $t_*$, then $t_*$ is non-degenerate.
\end{proposition}
\begin{proof}
Suppose that $\ker\opA'(t_*)\neq \{0\}$ and that $\opA$ is locally strongly monotone near $t_*$. Then for any $0\neq r\in\ker\opA'(t_*)$ sufficiently close to 0, we have
$$
C_{\mathrm{SM}}(\delta) \|r\|_\VV^2\le \vdua{\opA(t_*+r)}{r} = \frac{1}{2} \vdua{ \opA''(t_*)(r,r)}{r} + o(\|r\|_\VV^4).
$$
Rescaling $r$ by $\alpha>0$ small, and letting $\alpha\to 0$ we obtain that $C_{\mathrm{SM}}=0$, a contradiction.
\end{proof}

\begin{remark}\label{basrmk}
When applying topological degree theory, we will view $\opA:\VV\to\VV^*$ as a mapping $\RR^n\to\RR^n$ by identifying $\VV$ and $\VV^*$ with $\RR^n$. 
Following \cite[Section 1.3]{dinca2021brouwer}, we fix a basis $\{\tau_\alpha\}_{\alpha\in\Xi(G)}$ of $\VV$ and define the linear homeomorphism $h:\VV\to\RR^n$ with
$$
\VV\ni t=\sum_{\alpha\in\Xi(G)} \wh{t}_\alpha \tau_\alpha \mapsto h(t)=\sum_{\alpha\in\Xi(G)} \wh{t}_\alpha e_\alpha\in\RR^n,
$$
where $\{e_\alpha\}_{\alpha\in\Xi(G)}$ is the standard (ordered) basis in $\RR^n$.
Also, fix a basis $\{\tau_\alpha^*\}_{\alpha\in\Xi(G)}$ of $\VV^*$ and define the linear homeomorphism $g:\VV^*\to\RR^n$ analogously.
Then 
$$
\wh{\opA}:=g\circ \opA\circ h^{-1}:\RR^n\to\RR^n
$$
gives the desired mapping. Now suppose that two other bases $\{\wt{\tau}_\alpha\}_{\alpha\in\Xi(G)}\subset\VV$
and $\{\wt{\tau}^*_\alpha\}_{\alpha\in\Xi(G)}\subset\VV^*$ are given and let $\wt{h}:\VV\to\RR^n$ and $\wt{g}:\VV^*\to\RR^n$ be the corresponding linear homeomorphism.
But then 
$$
g^{-1}\circ g \circ \opA \circ h^{-1} \circ h = \opA = \wt{g}^{-1}\circ \wt{g} \circ \opA \circ \wt{h}^{-1} \circ \wt{h},
$$
which implies $\wt{\opA}:=\wt{g}\circ\opA\circ\wt{h}^{-1}=m\circ\wh{\opA}\circ\wt{m}$, where $m=\wt{g}\circ g^{-1}:\RR^n\to\RR^n$ and $\wt{m}=h\circ\wt{h}^{-1}:\RR^n\to\RR^n$. 
Using \cite[Lemma 1.3.1]{dinca2021brouwer} (\cite[Lemma 6.1.1]{dinca2009brouwer}), we obtain
$$
\deg( \wt{\opA}, \wt{h}(D), \wt{g}(0) ) = (\sgn \det m) (\sgn \det\wt{m}) \deg( \wh{\opA}, h(D), g(0) )
$$
for any open and bounded set $D\subset\VV$ with $0\not\in\opA(\partial D)$. We can conclude that the topological degree is independent of the choice of the basis
if $\VV$ and $\VV^*$ are oriented the same.
\end{remark}

Next, we determine the topological index of a zero of $\opA$.
The fact that the topological index of $t_*$ is related to its CC energy $\ene_\cc(t_*)$ 
and the eigenvalues of the operator $\wh{\ham_K}(t_*)_\VC$ is interesting on its own right.

\begin{theorem}[Index formula for SRCC---non-degenerate case]\label{indxformula}
Let $t_*$ be a zero of the CC mapping $\opA:\VV\to\VV^*$.
Then $t_*$ is non-degenerate if and only if  $\ene_{\cc}(t_*)\not\in\sigma(\wh{\ham_K}(t_*)_\VC)$, and in this case $t_*$ is an isolated zero and the topological index of $\opA$ at $t_*$ is given by
$$
i(\opA,t_*)=(-1)^{\nu},
$$
where 
$$
\nu=|\{j : \ene_j(\wh{\ham_K}(t_*)_\VC)\in\RR, \; \ene_j(\wh{\ham_K}(t_*)_\VC)< \ene_{\cc}(t_*) \}|.
$$
\end{theorem}
\begin{proof}
It is trivial to see that if $t_*$ is non-degenerate, then it is isolated: assume that $\ker\opA'(t_*)=\{0\}$ and write
\begin{equation}\label{isotaylor}
\dua{\opA(t_*+r)}{\opA'(t_*)r}_{\VV^*}=\|\opA'(t_*)r\|_{\VV^*}^2 + o(\|r\|_\VV^3),
\end{equation}
for all $\|r\|_\VV=\epsilon$, where $\epsilon>0$ is sufficiently small.
 This implies that $\opA(t_*+r)\neq 0$ for all $0<\|r\|_\VV<\epsilon$.

We can apply \cref{topvecindx} with the mappings $h:\VV\to\RR^n$ and $g:\VV^*\to\RR^n$ defined in \cref{basrmk}.
Using the notations of the said remark and \cref{SRCCDZ}, we have
\begin{align*}
\dua{\wh{\opA}'(h^{-1}(t_*))e_\alpha}{e_\beta}_{\RR^n}&=\dua{g\opA'(t_*)h^{-1}(e_\alpha)}{e_\beta}_{\RR^n}=\vdua{\opA'(t_*)h^{-1}(e_\alpha)}{g^\dag(e_\beta)}\\
&=\hdua{(\widehat{\ham_K}(t_*) - \ene_{\cc}(t_*))U_\alpha \Phi_0}{V_\beta\Phi_0},
\end{align*}
where $u_\alpha=h^{-1}(e_\alpha)$ and $v_\beta=g^\dag(e_\beta)$ and $\alpha,\beta\in\Xi(G)$. Here, $g^\dag:\RR^n\to \VV$ is the adjoint of $g$.
Therefore, using an appropriate basis transformation
$$
i(\opA,t_*)=\sgn \det \wh{\opA}'(h^{-1}(t_*))=\sgn \Bigg(\prod_{j\ge 0} (\ene_j(\widehat{\ham_K}(t_*)_\VC) - \ene_{\cc}(t_*))\Bigg).
$$
The proof is completed by noting that the elements of the matrix $\widehat{\ham_K}(t_*)_\VC$ are real, so its complex eigenvalues come in conjugate pairs, hence
only real eigenvalues contribute to the product above.
\end{proof}

\begin{proposition}\label{smindex}
If $\opA$ is locally strongly monotone near a zero $t_*$, then we have $i(\opA,t_*)=1$.
\end{proposition}
\begin{proof}
We have that in particular $\wh{\opA}:\RR^n\to\RR^n$ is monotone near $h(t_*)$, so according to \cref{monothm},
$\wh{\opA}$ is orientation-preserving near $h(t_*)$. But then \cref{opthm} implies that $i(\opA,t_*)=i(\wh{\opA},h(t_*))>0$.
\end{proof}

Using the ``unlinked'' form, we can determine the topological index in the Full CC case.
Recall that the eigenvalues $\ene_n(\ham_K)$, $n=0,1,\ldots$, are assumed to be increasingly ordered.
\begin{theorem}[Index formula for FCC---non-degenerate case]\label{fccnondeg}
Let $\VV=\VV(G^\full)$ and assume that $t_*\in\VV$ is a zero of the FCC mapping $\opA:\VV\to\VV^*$.
Then $e^{T_*}\Phi_0\in\HC^2$ is an (intermediately normalized) eigenfunction corresponding to some non-degenerate eigenvalue $\ene_\nu(\ham_K)$ if and only if
 $t_*$ is non-degenerate, and in this case $i(\opA,t_*)=(-1)^\nu$.
\end{theorem}
\begin{proof}
First, note that $\VV(G^\full)$ is rank-regular so $\wh{\ham_K}(t)_\VC=\ham_K(t)_\VC$.
We have $\ene_\cc(t_*)=\ene_\nu(\ham_K)$ by the equivalence of FCC and FCI (see Theorem 2.2 of Part I).

According to \cref{fcclemm}, we have that $e^{T_*}\Phi_0$ is a non-degenerate intermediately normalized eigenfunction if and only if 
 $\ene_{\cc}(t_*)\not\in\sigma(\ham_K(t_*)_\VC)$. In fact, $\ene_{\cc}(t_*)\in\sigma(\ham_K(t_*)_\VC)$ if and only if 
  there exists $R\Phi_0\in\VC$ nonzero, such that
$$
\dua{e^{-T_*}\ham_K e^{T_*}R\Phi_0}{S\Phi_0}=\ene_{\cc}(t_*) \dua{R\Phi_0}{S\Phi_0},
$$
for all $s\in\VV$. Since $\VV(G^\full)$ is excitation complete, according to \cref{unlinkedcc} the preceding equation is equivalent to
\begin{equation}\label{fccdeg}
\dua{\ham_K R e^{T_*}\Phi_0}{S\Phi_0}=\ene_{\cc}(t_*) \dua{R e^{T_*}\Phi_0}{S\Phi_0},
\end{equation}
for all $s\in\VV$. But this precisely means that the FCI eigenstate $e^{T_*}\Phi_0$ is degenerate, because $R e^{T_*}\Phi_0$ is another
eigenvector corresponding to the same eigenvalue $\ene_{\cc}(t_*)$.

We also conclude from \cref{hamspec} that $\sigma(\ham_K(t_*)_\VC)=\sigma(\ham_K)\setminus\{\ene_\cc(t_*)\}$.
Applying \cref{indxformula}, we obtain that $t_*$ is non-degenerate and $i(\opA,t_*)=(-1)^\nu$. 
\end{proof}

It is worth noting that, in the FCC case, the zero $t_*$ representing the intermediately normalized, non-degenerate \emph{ground state} (i.e. $\ene_\cc(t_*)=\ene_0(\ham_K)$)
%and   $e^{T_*}\Phi_0$)
has $i(\opA,t_*)=1$. Note that this is not necessarily true in the truncated case.
While the CC method is most commonly aimed at the ground state, it can also be used to find other intermediately normalized eigenfunctions as well.
Furthermore, it can also be used to obtain eigenfunctions which are orthogonal to the reference $\Phi_0$ according to the remark below.

\begin{remark}\label{srcceom}
The \emph{Equation-of-Motion Coupled-Cluster} (EOM-CC) method \cite{geertsen1989equation} is aimed at calculating \emph{excited} energies and states 
(i.e. $\ene_n(\ham_K)$ for $n>0$, and the corresponding eigenvectors) based on a CC ground-state solution. 
This is done in two steps. Let $\VV=\VV(G^\full)$.
 Firstly, a conventional CC calculation determines the ground state $\Psi=e^{T_*}\Phi_0$ such that $\opA(t_*)=0$, i.e. $\ham_K\Psi=\ene\Psi$.
Secondly, the targeted excited state is of the form
$\Psi_{\mathrm{ex}}=(r_0 I +R)e^{T_*}\Phi_0$,
 where $R$ is a cluster operator, see \cref{fcclemm}. We have
\begin{equation}\label{eomsimp}
\ham_K(t_*)_\VC R\Phi_0 = \ene_{\mathrm{ex}} R\Phi_0.
\end{equation}
In other words, we need to solve the eigenproblem of the projected similarity-transformed Hamiltonian $\ham_K(t_*)_\VC$.
Furthermore, similarly to the proof of \cref{unlinkedcc}, it is easy to see that $\opA(t_*)=0$ implies 
$$
\dua{R\ham_K(t_*) \Phi_0}{S\Phi_0} = \ene_\cc(t_*) \dua{R\Phi_0}{S\Phi_0}\quad\text{for all}\quad s\in\VV.
$$
Subtracting this from \cref{eomsimp}, we get the ``commutator form'' of the EOM-CC equation:
\begin{equation}\label{eomcomm}
\dua{[\ham_K (t_*),R]\Phi_0}{S\Phi_0}= \Delta\ene\dua{R\Phi_0}{S\Phi_0}, \quad \text{where} \quad \Delta\ene=\ene_{\mathrm{ex}}-\ene,
\end{equation}
and $\ene=\ene_\cc(t_*)$ is the ground-state energy as given by the CC method.
Let $\VV$ be an arbitrary amplitude space. Recalling the expression \cref{SRCCD} for $\opA'(t)$, we can rephrase the EOM-CC equation \cref{eomcomm} as \emph{the weak eigenvalue problem for} $\opA'(t_*):\VV\to\VV^*$ (cf. \cite[Section 13.6.3]{helgaker2014molecular}), i.e.
\begin{equation}\label{eomder}
\vdua{\opA'(t_*)r_j}{s}=\Delta\ene_j \vdua{r_j}{s},
\end{equation}
for $j=1,\ldots,J$\footnote{Here, $R_j$ is not to be confused with the rank-decomposition (3.6) of Part I.} and $s\in\VV$ is arbitrary.\footnote{A similar
relation holds if $t_*$ does not represent the ground state.}
Notice that the $\Delta\ene_j$'s are in general complex. 
 Using \cref{indxformula} we can obtain the following.
Suppose that $\Delta\ene_1,\ldots,\Delta\ene_\mu$ are given by \cref{eomder} and are all nonzero. Then
\begin{equation}\label{indxeom}
i(\opA,t_*)=(-1)^\nu, \quad \nu=|\{j : \Delta\ene_j\in\RR, \; \Delta\ene_j<0\}|.
\end{equation}
Due to the \emph{nonvariational property} of truncated CC (see Section 2.3 of Part I), it is not \emph{a priori} clear whether the (real) excited energies are higher than the
 ground-state energy,
i.e. whether $\Delta\ene_j>0$.
Therefore, \cref{indxeom} quantifies this nonvariational property through the topological index $i(\opA,t_*)$.
\end{remark}

Next, we draw a connection between the degeneracy of a zero $t_*$ and the Fock-splitting \cref{fockfluc} of the Hamiltonian.
Define $\omega_0(t_*)=\dua{\mathcal{W}_K(t_*)\Phi_0}{\Phi_0}$, which is the CC correction to the lowest eigenvalue $\Lambda_0$ of $\mathcal F$, so that the CC energy at $t_*$ is 
obtained as $\ene_\cc(t_*)=\Lambda_0+\omega_0(t_*)$.

\begin{proposition}\label{focknondeg}
Let $\VV(G)$ be a rank-regular amplitude space and $t_*$ a zero of $\opA$. 
%Let $\omega_0=\dua{\mathcal{W}_K(t_*)\Phi_0}{\Phi_0}$ and 
Define the linear operator $\mathcal{Q}(t_*):\VC\to\VC$ via its matrix in the Slater determinant basis as
$$
[\mathcal{Q}(t_*)]_{\alpha\beta}=\epsilon_\alpha \delta_{\alpha\beta} + \sum_{\gamma\in\Xi(G)} t_{*,\gamma}\epsilon_\gamma \dua{X_\gamma\Phi_{\beta}}{\Phi_\alpha}
\quad\text{for all}\quad\alpha,\beta\in\Xi(G).
$$
Then $\omega_0(t_*)\not\in\sigma( \mathcal{Q}(t_*) + \mathcal{W}_K(t_*)_\VC)$
is equivalent to $\ene_{\cc}(t_*)\not\in\sigma(\ham_K(t_*)_\VC)$, i.e. to the fact that $t_*$ is a non-degenerate zero of $\opA$.
\end{proposition}
\begin{proof}
We have using \cref{fockcomm},
$$
\begin{aligned}
\ene_\cc(t_*)&=\dua{e^{-T_*}\mathcal{F}_Ke^{T_*}\Phi_0}{\Phi_0} + \dua{\mathcal{W}_K(t_*)\Phi_0}{\Phi_0}\\
&=\dua{\mathcal{F}_K\Phi_0}{\Phi_0} + \dua{[\mathcal{F}_K,T_*]\Phi_0}{\Phi_0}+\dua{\mathcal{W}_K(t_*)\Phi_0}{\Phi_0}\\
&=\Lambda_0 + \dua{\mathcal{W}_K(t_*)\Phi_0}{\Phi_0}=\Lambda_0 + \omega_0(t_*).
\end{aligned}
$$
Similarly, 
$$
\begin{aligned}
\dua{\ham_K(t_*)\Phi_\beta}{\Phi_\alpha}&=\dua{\mathcal{F}_K\Phi_\beta}{\Phi_\alpha} + \dua{[\mathcal{F}_K,T_*]\Phi_\beta}{\Phi_\alpha} + \dua{\mathcal{W}_K(t_*)\Phi_\beta}{\Phi_\alpha}\\
&= (\Lambda_0 + \epsilon_\alpha) \delta_{\alpha\beta} + \sum_{\gamma\in\Xi(G)} t_{*,\gamma}\epsilon_\gamma \dua{X_\gamma\Phi_{\beta}}{\Phi_\alpha} + \dua{\mathcal{W}_K(t_*)\Phi_\beta}{\Phi_\alpha}.
\end{aligned}
$$
Then, in the Slater determinant basis
$$
\ham_K(t_*)_\VC= \Lambda_0 I + \mathcal{Q}(t_*) + \mathcal{W}_K(t_*)_\VC.
$$
Hence, $\ene_{\cc}(t_*)\not\in\sigma(\ham_K(t_*)_\VC)$ is equivalent to $\omega_0(t_*)\not\in\sigma( \mathcal{Q}(t_*) + \mathcal{W}_K(t_*)_\VC)$,
which finishes the proof.
\end{proof}

We now consider the case of a degenerate zero. Clearly, if $r\in\ker\opA'(t_*)\neq\{0\}$, we have to consider higher-order terms of the Taylor polynomial of
$\opA$ at $t_*$,
$$
\opA(t_*+r)=\opA'(t_*)r + \frac{1}{2}\opA''(t_*)(r,r) + \mathcal{R}_3(t_*;r),
$$
where $\opA''(t_*) : \VV\times\VV\to\VV^*$ is a bounded bilinear mapping.
Here, we only consider the second-order information.

Assume from now on that $\VV$ is rank-regular, so that $\wh{\ham_K}(t)_\VC=\ham_K(t)_\VC$.
Suppose that $\ene_{\cc}(t_*)\in\sigma(\ham_K(t_*)_\VC)$ and that $R_1\Phi_0,\ldots,R_\mu\Phi_0\in\VC$ are the \emph{right}
 eigenvectors of $\ham_K(t_*)_\VC$ corresponding to $\ene_{\cc}(t_*)$, 
$$
\dua{\ham_K(t_*) R_j\Phi_0}{S\Phi_0}=\ene_{\cc}(t_*)\dua{R_j\Phi_0}{S\Phi_0} \quad\text{for all $j=1,\ldots,\mu$ and all $s\in\VV$.}
$$
Also, suppose that $L_1\Phi_0,\ldots,L_\mu\Phi_0\in\VC$ are the \emph{left}
eigenvectors of $\ham_K(t_*)_\VC$ corresponding to $\ene_{\cc}(t_*)$,
$$
\dua{\ham_K(t_*)^\dag L_j\Phi_0}{S\Phi_0}=\ene_{\cc}(t_*)\dua{L_j\Phi_0}{S\Phi_0}\quad\text{for all $j=1,\ldots,\mu$ and all $s\in\VV$.}
$$
The corresponding right-, and left eigenspaces are 
\begin{align*}
	W_R&=\ker \opA'(t_*)=\Span\{r_1,\ldots,r_\mu\},\\
	W_L&=\ker \opA'(t_*)^\dag=\Span\{\ell_1,\ldots,\ell_\mu\},
\end{align*}
and let $Q:\VV\to\VV$ be the orthogonal projector onto $W_L$. Further, define $\wh{Q}:\VC\to\VC$ via $\dua{\wh{Q}U\Phi_0}{V\Phi_0}=\dua{Qu}{v}$ for all $u,v\in\VV$.
We introduce the mapping $\mathcal{B}:\VV\to\VV^*$ via
\begin{equation}\label{degB}
\vdua{\mathcal{B}(t)}{s}=\frac{1}{2} \hdua{\wh{Q}[[\ham_K(t_*),T],T]\Phi_0}{S\Phi_0}=\frac{1}{2} \hdua{\wh{Q}[[\mathcal{W}_K(t_*),T],T]\Phi_0}{S\Phi_0},
\end{equation}
that is, the $Q$-projection of $\frac{1}{2}\opA''(t_*)(t,t)$. In the second equality, we used \cref{simcomm}.
Also, note that $\mathcal{B}$ is homogeneous of degree 2, i.e. $\mathcal{B}(\alpha t)=\alpha^2 \mathcal{B}(t)$.
The next theorem follows esentially from Leray's second reduction formula (\cref{leraysec}).

\begin{theorem}
[Index formula for SRCC---degenerate case]\label{indxformuladege}
Let $t_*$ be zero of the CC mapping $\opA:\VV\to\VV^*$.
Suppose that $\ene_{\cc}(t_*)\in\sigma(\ham_K(t_*)_\VC)$ and let $W_R$, $W_L$ and $Q$ be as above. 
Assume that
\begin{equation}\label{indxdegcond}
	\mathcal{B}(t)\neq 0\quad \text{for all} \quad t\in \partial B_\VV(0,1).
\end{equation}
Then, $t_*$ is an isolated zero and the topological index of $\opA$ at $t_*$ is given by
$$
i(\opA,t_*)=i(\opA'(t_*)+Q,0)\,i(\mathcal{B}|_{W_R},0).
$$
\end{theorem}
\begin{proof}
First, we prove that $t_*$ is isolated. When $r\not\in\ker\opA'(t_*)$ and small, it follows that $\opA(t_*+r)\neq 0$ similarly to \cref{isotaylor}.
If, however $r\in\ker\opA'(t_*)$, then we may write
$$
\dua{\opA(t_*+r)}{\opA''(t_*)(r,r)}_{\VV^*}=\dua{\underbrace{\opA'(t_*)r}_{0}}{\opA''(t_*)(r,r)}_{\VV^*} + \frac{1}{2} \|\opA''(t_*)(r,r)\|_{\VV^*}^2 + \mathcal{O}(\|r\|_\VV^5)
$$
for all $r\in B(0,\epsilon)$ for sufficiently small $\epsilon>0$. Condition \cref{indxdegcond} implies that $\opA(t_*+r)\neq 0$ for all $r\in B_\VV^*(0,\epsilon)$.
	
Next, we apply \cref{noninvpert} with the choice $D=B_\VV(0,\epsilon)$, $L=\opA'(t_*)$ and 
$$
\vdua{\mathcal{N}(t,\lambda)}{s}=\sum_{k=2}^{2N} \frac{\lambda^{k-2}}{k!} \hdua{[\ham_K(t_*),T]_{(k)}\Phi_0}{S\Phi_0},
$$
where we used \cref{bchdoub}.
Because $\Ran Q=\ker \opA'(t_*)^\dag$, it follows that 
$$
\ker Q=(\Ran Q)^\perp=(\ker \opA'(t_*)^\dag)^\perp=\Ran \opA'(t_*)=\Ran L.
$$
Moreover, since $t_*$ is an isolated zero, it is possible to choose $\delta>0$ so that the equation 
$$
\lambda^{-1}\opA(t_*+\lambda t)=\opA'(t_*)t + \lambda \mathcal{N}(t,\lambda)=0
$$
does not admit a solution $t\in \partial B_\VV(0,\delta)$ for any $\lambda\in(0,1]$.

Note that $Q\mathcal{N}(t,0)=\mathcal{B}(t)\neq 0$ for all $t\in B_\VV^*(0,\delta)$ by assumption \cref{indxdegcond} and the homogenity of $\mathcal{B}$.
We see that conditions (i) and (ii) of \cref{noninvpert} are satisfied and the result follows.
\end{proof}

The preceding theorem reduces the computation of the index to a low-dimensional problem but the zero is still degenerate. In fact, since
$$
\vdua{\mathcal{B}'(t)u}{v}=\frac{1}{2}\hdua{\wh{Q}([[\ham_K(t_*),U],T] + [[\ham_K(t_*),T],U])\Phi_0}{V\Phi_0},
$$
we have that $t=0$ is a degenerate zero of $\mathcal{B}|_{W_R}$ and by assumption the only zero.
Therefore, we need to apply \cref{croninperturb} to determine $i(\mathcal{B}|_{W_R},0)$.

\begin{corollary}\label{degesingle}
For a degenerate, isolated zero $t_*$ of $\opA$, $\dim W_R=1$,
and for which \cref{indxdegcond} holds, we have $i(\opA,t_*)=0$.
\end{corollary}
\begin{proof}
Let $W_R=\Span\{r\}$ and $W_L=\Span\{\ell\}$. 
We apply \cref{croninperturb} to the mapping $\mathcal{B}|_{W_R}$ with $D=B_\VV(0,\epsilon)$, $\epsilon>0$ arbitrary.
Fix $z'=\eta \ell\in\Ran Q=W_L$ such that $0<|\eta|<\delta$. Since $t=cr\in W_R$, the equation
$\mathcal{B}(t)=z'$ in $B_\VV(0,\epsilon)\cap W_R$ is equivalent to finding $0\neq c\in\RR$ such that
\begin{equation}\label{degesingeq}
\frac{c^2}{2} \dua{(\ham_K(t_*)-\ene_\cc(t_*)) R^2\Phi_0}{L\Phi_0}=\eta\dua{L\Phi_0}{L\Phi_0}.
\end{equation}
 Note that the inner product on the left-hand side is nonzero by assumption \cref{indxdegcond}.
 Choose $\eta$ to be of opposite sign as the inner product on the left-hand side. 
Then there are no solutions $c$, so $i(\mathcal{B}|_{W_R},t_*)=0$ and therefore $i(\opA,t_*)=0$ by \cref{indxformuladege}.
\end{proof}

\begin{corollary}
Let $t_*$ be an isolated zero of the CC mapping $\opA:\VV\to\VV^*$.
Suppose that $\ene_{\cc}(t_*)\in\sigma(\ham_K(t_*)_\VC)$ and let $W_R$ as above. Assume that $\ker\mathcal{B}'(t)=\{0\}$ for all $0\neq t\in W_R$.
If $\mu:=\dim W_R$ is odd, then $i(\opA,t_*)=0$ and if $\mu$ is even, then $i(\opA,t_*)$ is even. In particular, under the
above hypotheses, $i(\opA,t_*)$ cannot be $\pm 1, \pm 3, \pm 5,\ldots$
\end{corollary}
\begin{proof}
Let $z'\in B_\VV^*(0,\delta)\cap W_L$. Then the equation $\mathcal{B}(t)=z'$ for $0\neq t\in W_R$ is equivalent to
$$
\frac{1}{2}\dua{[[\ham_K(t_*),T],T]\Phi_0}{L\Phi_0}=\dua{Z'\Phi_0}{L\Phi_0},
$$
for all $\ell\in W_L$. Let $\mathcal{T}$ denote the set of solutions $t$ of the preceding equation (which can be empty). 
By \cref{croninperturb}, $|\mathcal{T}|=m$ for some $m$ finite. Notice that $\mathcal{T}$ is closed under the operation $t\mapsto -t$,
so $m$ is even (we also used that $0\not\in\mathcal{T}$), and let
$$
\mathcal{T}=\{t_1,\ldots,t_{\frac{m}{2}},-t_1,\ldots,-t_{\frac{m}{2}}\}
$$
using some appropriate indexing.
From the linearity of $t\mapsto \mathcal{B}'(t)$, we get via \cref{topvecindx},
\begin{equation}\label{degdegreal}
\begin{aligned}
\deg(\mathcal{B}|_{W_R}, B_\VV(0,r)\cap W_R, z')&=\sum_{i=1}^{\frac{m}{2}} \sgn\det g\mathcal{B}'(t_i)h^{-1}+\sgn\det g\mathcal{B}'(-t_i)h^{-1}\\
&=(1+(-1)^\mu) \sum_{i=1}^{\frac{m}{2}} \sgn\det g\mathcal{B}'(t_i)h^{-1},
\end{aligned}
\end{equation}
for some sufficiently large $r>0$.
\end{proof}

We close this section with two remarks.

\begin{remark}\label{smdegremark}
\leavevmode
\begin{enumerate}[label=\normalfont(\roman*)]
\item Note that, according to \cref{croninperturb}, a zero $t_*$ of topological index 0 is ``numerically unstable'',
because one could miss zeros altogether if the equations are solved with finite precision arithmetic, even in the FCC case. Therefore,
the degenerate zeros of the SRCC mapping $\opA$ are not robust in general. We have already seen that in the FCC case, the CC energy $\ene_\cc(t_*)$ of a degenerate zero $t_*$ is
a degenerate eigenvalue $\ene$ of the Hamiltonian (see \cref{fccdegr}). Thus, we can conclude that the SRCC method is in general unsuitable for finding degenerate eigenstates
and eigenvalues---an empirical fact that is well known among the practitioners of the SRCC method.

\item \Cref{smindex} and the preceding calculations imply that any approach that stipulates the local strong monotonicity of $\opA$ near $t_*$ can only provide an incomplete description of the SRCC method.
\end{enumerate}
\end{remark}

\subsection{Local properties---complex case}\label{secsrcccmplx}

We now discuss what happens when \emph{complex} amplitude spaces are considered instead. We explain the complex case in detail, because the differences from the
real case are somewhat subtle. 
Assume that $\VV$ is a complex amplitude space and let $\VV^*$ denote its anti-dual.
  It is clear that $t\mapsto\dua{\opA(t)}{s}$ is a (complex) polynomial for fixed $s\in\VV$, hence with the appropriate identifications,
$\opA_\CC:\VV\to\VV^*$ is a holomorphic mapping, where we used the subscript $\CC$ to highlight the difference.\footnote{Note that the realification of $\opA_\CC$, 
$(\opA_\CC)_\RR$  does \emph{not} equal $\opA$ (they are mappings of different type: $(\opA_\CC)_\RR : \RR^{2n}\to\RR^{2n}$).}
Of course, a real zero $t_*\in\VV$ to $\opA(t_*)=0$ is automatically a ``complex'' zero: $\opA_\CC(t_*)=0$.
Further, using the fact that the Hamiltonian is real (by which we mean $\dua{\ham_K\Phi_\alpha}{\Phi_\beta}\in\RR$),
 $\opA_\CC(t_*)=0$ if and only if $\opA_\CC(\ol{t}_*)=0$. Also, $\ene_\cc(\ol{t})=\ol{\ene_\cc(t)}$.

From \cref{cmplxdeg} (i) we immediately get that $\deg(\opA_\CC,U,0)\ge 0$ for every bounded open $U\subset\VV$.
In particular, $i(\opA_\CC,t_*)\ge 0$ for every isolated zero $t_*$. Notice that, even if $t_*$ is a zero of both $\opA$ and $\opA_\CC$, its real and complex
indices $i(\opA,t_*)$
and $i(\opA_\CC,t_*)$ may differ; for instance we know that $i(\opA,t_*)$ can have a sign, while $i(\opA_\CC,t_*)$ cannot.
Also, \cref{cmplxdeg} (ii) implies that $i(\opA,t_*)\ge 2$ for an isolated, degenerate zero $t_*$.
Moreover, the following is true.
\begin{theorem}
If $t_*\in\VV$ is a \emph{real} isolated zero of both $\opA$ and $\opA_\CC$, then
$$
|i(\opA,t_*)|\le i(\opA_\CC,t_*), \quad i(\opA,t_*)\equiv i(\opA_\CC,t_*) \mod 2.
$$
\end{theorem}
\begin{proof}
The result follows from \cref{holodegrel} with the choice $D=B_\VV(t_*,\delta)$.
\end{proof}

For simplicity, we assume from now on that $\VV$ is rank-regular (\cref{rankreg}), so that $\wh{\ham_K}(t)_\VC=\ham_K(t)_\VC$.
Adapting the proofs of \cref{indxformula} and \cref{indxformuladege} in the complex case, we have
\begin{theorem}[Index formula for SRCC---complex case]\label{indxcomplx}
Let $t_*$ be an isolated zero of $\opA_\CC:\VV\to\VV^*$. Then the following hold true.
\begin{enumerate}[label=\normalfont(\roman*)]
\item The zero $t_*$ is non-degenerate if and only if $\ene_{\cc}(t_*)\not\in\sigma(\ham_K(t_*)_\VC)$, and in this case $i(\opA_\CC,t_*)=1$.
\item  Suppose that $\ene_{\cc}(t_*)\in\sigma(\ham_K(t_*)_\VC)$ and that the second-order regularity assumption \cref{indxdegcond} holds true.
Then
$$
i(\opA_\CC,t_*)=i(\mathcal{B}|_{W_R},0).
$$
\item Suppose that $\ene_{\cc}(t_*)\in\sigma(\ham_K(t_*)_\VC)$ and that the second-order regularity assumption \cref{indxdegcond} holds true.
 Assume further, that $\ker\mathcal{B}'(t)=\{0\}$ for all $0\neq T\Phi_0 \in W_R$. Then $i(\opA_\CC,t_*)=m\ge 2$, 
where $m$ is the number of solutions $0\neq R\Phi_0\in W_R$
to the equation
$$
\dua{(\ham_K(t_*)-\ene_\cc(t_*)) R^2\Phi_0}{L\Phi_0}=\dua{Z\Phi_0}{L\Phi_0}\quad (L\Phi_0\in W_L)
$$
for any $Z\Phi_0\in W_L\cap B^*(0,\delta)$ for sufficiently small $\delta>0$.
\end{enumerate}
\end{theorem}
\begin{proof}
For (i), it is enough to note that $\ker\opA_\CC'(t_*)\neq \{0\}$ follows via the same calculation as in the proof of \cref{indxformula}.
Part (ii) follows since $i(\opA_\CC'(t_*)+Q)=1$.
For the proof of (iii), notice that \cref{degdegreal} now reads
\begin{align*}
\deg(\mathcal{B}|_{W_R}, B_\VV(0,r)\cap W_R, z)&=\sum_{i=1}^{m} \sgn |\det g\mathcal{B}'(t_i)h^{-1}|^2=m.\qedhere
\end{align*}
\end{proof}

\begin{corollary}
For a degenerate, isolated zero $t_*$ of $\opA_\CC$, $\dim W_R=1$,
and for which \cref{indxdegcond} holds, we have $i(\opA_\CC,t_*)=2$.
\end{corollary}
\begin{proof}
The proof is analogous to that of \cref{degesingle}, but \cref{degesingeq} always has exactly two nonzero complex solutions $c$.
\end{proof}

We close this section with a few remarks.

\begin{remark}
\leavevmode
\begin{enumerate}[label=\normalfont(\roman*)]
\item Luckily, for a real zero $t_*\in\VV$, the condition $\ene_{\cc}(t_*)\not\in\sigma(\ham_K(t_*)_\VC)$ is formally the same as in the real case,
therefore a non-degenerate real zero is automatically a non-degenerate zero of $\opA_\CC$.

\item The degeneracy of a complex zero $t_*$ manifests itself in numerical computations as follows. 
Suppose the hypotheses of \cref{indxcomplx} (iii) hold true.
 Combining this with \cref{holocronin}, we get that the perturbed equation $\opA(t)=z'$ has \emph{exactly} $m$ solutions for almost all $z'\in\VV$ sufficiently close to zero. 
 This is in contrast with the real case, when one might completely ``lose'' solutions when the index is zero (\cref{smdegremark} (i)).
 The appearance of multiple complex zeros in degenerate situations was conjectured based on numerical observations in \cite{piecuch1990coupled,piecuch2000search}.

\item If the Hamiltonian is real (see above), then the index of the complex conjugate zero is the same: $i(\opA_\CC,t_*)=i(\opA_\CC,\ol{t}_*)$.

\item The classical \emph{B\'ezout theorem} states that if a polynomial system
\[
\left.
\begin{aligned}
P_1(x_1,\ldots,x_d)&=0\\
P_2(x_1,\ldots,x_d)&=0\\
\ldots&\\
P_n(x_1,\ldots,x_d)&=0
\end{aligned}
\right\}
\]
has a finite number of zeros in $\CC^d$, then the number of zeros (counting multiplicities) is at most $\Delta=\Delta_1\cdots \Delta_d$ (called the \emph{B\'ezout number}),
where $\Delta_k$ denotes the degree of $P_k$. As remarked earlier, according to the Baker--Campbell--Hausdorff expansion \cref{bchagain}, for the polynomials
constituting the system $\opA(t)=0$ there holds $\Delta_1=\ldots=\Delta_d=4$, hence the B\'ezout number of the CC equations is $\Delta=4^d$, where $d=\dim\VV$.
This is typically a huge number. However, it is known that the B\'ezout number often grossly overestimates the number of zeros. 
In fact, it was observed numerically that the number of zeros for the (truncated) CC equations is much less than the B\'ezout number \cite{piecuch2000search}.
\end{enumerate}
\end{remark}

\subsection{Continuation of solutions}\label{seccont}

In this section we discuss how solutions of different CC methods can be ``connected'' in a systematic way. The idea is not new to this field (and certainly not
new to nonlinear analysis, see e.g. \cite{zeidler1993nonlinear}), and it has been 
a subject of both theoretical and numerical investigations in the CC literature, as we have already mentioned in the introduction.

The main theoretical tool we use to describe the aformentioned connection is a specific type of homotopy.

\begin{definition}\label{admhomo}
Let $\VV^1$ be an amplitude space with direct sum decomposition $\VV^1=\VV^0\oplus\VV^\angle$.
Let $\opA^j : \VV^j\to(\VV^j)^*$ be continuous mappings for $j=0,1$. 
A continuous map $\mathcal{K} : \VV^1\times[0,1]\to(\VV^1)^*$ is said to be an \emph{admissible homotopy}, if
\begin{enumerate}[label=\normalfont(\roman*)]
\item $\dua{\mathcal{K}(t^1,0)}{s^0}=\dua{\opA^0(t^0)}{s^0}$ for all $t^1\in\VV^1$, $s^0\in\VV^0$, and
\item $\mathcal{K}(\cdot,1)=\opA^1$.
\end{enumerate}
Furthermore, an admissible homotopy $\mathcal{K} : \VV^1\times[0,1]\to(\VV^1)^*$ is said to be \emph{faithful}, if for every 
$t_{**}^0\in\VV^0$ such that $\opA^0(t_{**}^0)=0$, there exists $t_{**}^\angle\in\VV^\angle$ so that with $t_{**}^1=t_{**}^0+t_{**}^\angle\in\VV^1$, there holds $\mathcal{K}(t_{**}^1,0)=0$.
\end{definition}

\begin{example}
When $\opA^0$ is the projection of $\opA:=\opA^1$ onto $\VV^0$, a simple admissible homotopy can be given by
\begin{equation}\label{simphom}
\dua{\opK(t^1,\lambda)}{s^1}=\dua{\opA(t^0+\lambda t^\angle)}{s^0} + \dua{\opA(t^1)}{s^\angle},
\end{equation}
for all $t^1\in\VV^1$, $s^1\in\VV^1$ and $\lambda\in[0,1]$ (cf. \cref{kpcomp}).
\end{example}

Let $\VV^1=\VV(G^\full)$ and $\opA^1$ be the FCC mapping, $\VV^0$ some rank-truncated space and $\opA^0$ the truncated CC mapping.
This case is particularly important due to the equivalence of FCC and FCI (see Theorem 4.4 of Part I), so existence of a solution to the FCI problem (essentially the Schr\"odinger
equation) can be exploited to infer the existence of a \emph{truncated} CC solution. Furthermore, the topological index of the CC solution can be
determined by the results of \cref{secsrccop} and the homotopy invariance of the topological degree can be used relate these quantities in certain situations (see \cref{indexrel}
below).

The \emph{zero set} of an (admissible) homotopy $\opK$ is defined as
\begin{equation}\label{kpz}
\mathcal{Z}(\opK)=\{(t^1_*,\lambda)\in \VV^1\times[0,1] : \opK(t_*^1,\lambda)=0\}.
\end{equation}
We omit $\opK$ from the notation $\mathcal{Z}(\opK)$ whenever it is clear from the context.
The $\lambda$-sections of $\mathcal{Z}$ are denoted as $\mathcal{Z}_\lambda=\{ t_*^1\in\VV^1 : (t_*^1,\lambda)\in\mathcal{Z}\}$.
Clearly, $(\mathcal{Z}(\opK))_1=(\opA^1)^{-1}(0)$ for \emph{any} admissible homotopy $\opK$. Furthermore,  $\Pi_{\VV^0}(\mathcal{Z}(\opK))_0=(\opA^0)^{-1}(0)$
for any faithful, admissible homotopy $\opK$.
We will sometimes explicitly label $t_*^1$ with the $\lambda$ to which it corresponds as $t_*^1(\lambda)$.

Recall that any topological space can be partitioned into a family of \emph{connected components}, which are maximal (w.r.t. set inclusion), closed connected sets.
We are interested in the connected components of $\mathcal{Z}(\opK)$.
The Leray--Schauder continuation principle \cite[Theorem 2.1.3]{dinca2021brouwer} immediately implies the following.
\begin{theorem}\label{admcont}
Suppose that $\mathcal{D}\subset \VV^1\times[0,1]$ is a bounded open set and let $\opK:\VV\times[0,1]\to\VV^*$ be an admissible homotopy such that 
\begin{equation}\label{admnz}
\opK(t^1,\lambda)\neq 0\quad\text{for all}\quad(t^1,\lambda)\in\partial\mathcal{D}.
\end{equation}
Then the following statements hold true.
\begin{enumerate}[label=\normalfont(\roman*)]
\item $\deg(\opK(\cdot,0),\mathcal{D}_0,0)=\deg(\opA^1,\mathcal{D}_1,0)=:d$. 
\item If $d\neq 0$, then there is a connected component $\mathcal{C}$ of $\mathcal{Z}$, such that 
$$
\mathcal{C}\cap (\mathcal{Z}_j\times\{j\}) \neq\emptyset\quad\text{for}\quad j=0,1.
$$
\end{enumerate}
\end{theorem}
This general theorem can be used to prove existence results, which essentially consists
of establishing the boundary condition \cref{admnz}. Furthermore, note that (i) implies that 
\begin{equation}\label{indexrel}
\sum_{t_{**}^1\in \mathcal{Z}_0\cap\mathcal{D}_0} i(\opK(\cdot,0),t_{**}^1) = \sum_{t_{*}^1\in \mathcal{Z}_1\cap\mathcal{D}_1} i(\opA^1,t_{*}^1)=d.
\end{equation}
This is a necessary condition for zeros in $\mathcal{Z}_0$ and $\mathcal{Z}_1$ be in the same bounded connected component.

\begin{remark}
In the case when $\mathcal{D}$ is possibly \emph{unbounded}, one can invoke \cite[Theorem 2.1.4]{dinca2021brouwer} to obtain the following statement:
If $\mathcal{Z}_0$ is bounded and $\deg(\opK(\cdot,0),D,0)\neq 0$ for some bounded open set $D\supset \mathcal{Z}_0$,
then there is a connected component $\mathcal{C}$ of $\mathcal{Z}$ intersecting $\mathcal{Z}_0\times\{0\}$ which either intersects
$\mathcal{Z}_1\times\{1\}$ or is unbounded. The unboundedness of $\mathcal{C}$ has been observed numerically for a specific type of homotopy
and its significance is discussed in the next remark.
\end{remark}

\begin{remark}\label{homormrk}
In \cite{piecuch2000search}, a specific type of admissible homotopy (essentially \cref{simphom}, which will be discussed in \cref{seckp} below)
 is used as the basis to distinguish ``physical'' truncated solutions from ``unphysical'' ones.
Roughly speaking, they call a truncated solution  ``physical'' if it shares a \emph{bounded} connected component with an FCC solution (although they 
require more regularity of the connected component in question). They also found that some FCC (resp. truncated) solutions cannot be ``connected''
to any truncated (resp. FCC) solution. While this approach to the classification of solutions seems attractive at first, it has a serious conceptual drawback:
 it is unclear why one particular
homotopy is preferred over the (infinitely many) others. For instance, it is conceivable that an admissible homotopy classifies a truncated solution as ``unphysical'' while 
another homotopy classifies it as ``physical''. Moreover, their homotopy itself does not admit an obvious physical interpretation. 
We will not pursue this line of thought any further in this work, and view homotopies as purely mathematical devices.
\end{remark}

As an example of an admissible homotopy, we consider the \emph{linear homotopy}.
Let $\VV^0\subset\VV^1$ be any subspace (a typical choice is based on some truncation: $\VV^0=\VV(G(1,\ldots,\rho))$) and
$\VV^\angle:=(\VV^0)^\perp$, where the orthogonal complement is taken with respect to the $\VV$-inner product. 
Also, to emphasize that $\VV^\angle$ is actually given by the $\VV$-orthogonal complement of $\VV^0$, we shall write $t^1=t^0+t^\perp$ for some unique $t^0\in\VV^0$
and $t^\perp\in(\VV^0)^\perp$, for any $t^1\in\VV^1$.
Let $\opK_\LI : \VV^1\times[0,1] \to (\VV^1)^*$ be given by
$$
\dua{\opK_\LI(t^1,\lambda)}{s^1}= (1-\lambda)\left(\dua{\opA^0(t^0)}{s^0} + \alpha|\dua{t^\perp-u^\perp}{s^\perp}_\VV|\right) + \lambda \dua{\opA^1(t^1)}{s^1}
$$
for any $t^1,s^1\in\VV^1$ and $\lambda\in[0,1]$, and some fixed constants $\alpha>0$ and $u^\perp\in(\VV^0)^\perp$.
Then, clearly $\opK_\LI(\cdot,1)=\opA^1$. Further, $\opK_\LI(t^1_{**},0)=0$ is equivalent to
$\opA^0(t_{**}^0)=0$ and $t_{**}^\perp=u^\perp$. Therefore, $\opK_\LI$ is a faithful, admissible homotopy. Also, $\opK_\LI$ has the following trivial property:
if $t_*^1\in\VV^1$ is such that $\opA^0(t_*^0)=0$ and $\opA^1(t_*^1)=0$, then $\opK_\LI(t_*^1,\lambda)=0$ for all $\lambda\in[0,1]$.

As a simple application of the linear homotopy, we give a variant of the existence result \cite[Theorem 4.1]{rohwedder2013error}.

\begin{theorem}\label{smexi}
Suppose that $t_*^1\in\VV^1$ is a zero of $\opA^1$. Let $\kappa=\|t_*^\perp\|_\VV$. Assume the following hold true.
\begin{enumerate}[label=\normalfont(\roman*)]
\item $\opA^j:\VV^j\to\VV^j$ is $C^1$.
\item $L^0 \kappa=\sup_{\|u^0\|_\VV=1} \dua{\opA^0(t_*^0)-\opA^1(t_*^1)}{u^0}$ for some $L^0 < \infty$.
\item $\opA^j$ is strongly monotone in $B_{\VV^{j}}(t_*^j,\delta_j)$ with constant $C_{\mathrm{SM}}^j=C_{\mathrm{SM}}^j(\delta_j)>0$ for $j=0,1$.
\item $\kappa< \delta\frac{C_{\mathrm{SM}}^0}{L^0}$, where $\delta=\min\{\delta_0,\delta_1\}$
\end{enumerate}
Then there exists a unique $t_{**}^0\in B_{\VV^0}(t_*^0,\delta)$ such that $\opA^0(t_{**}^0)=0$. Furthermore, $i(\opK_\LI(\cdot,0),t_{**}^1)=i(\opA^0,t_*^0)=1$.
\end{theorem}
\begin{proof}
Set $\alpha:=C_{\mathrm{SM}}^0$ and $u^\perp:=t_*^\perp$ in the definition of $\opK_\LI$. 
We prove that the boundary condition $\opK_\LI(t^1,\lambda)\neq 0$ holds true for all $t^1\in\partial B_\VV(t_*^1,\delta)$ and $\lambda\in[0,1]$.
Write $t^1=t^1_* + r^1$, where $\|r^1\|_\VV=\delta$ and
$$
\dua{\opK_\LI(t^1_*+r^1,\lambda)}{r^1}=: (1-\lambda)A_0 +  \lambda A_1.
$$
Here,
\begin{align*}
A_0&=\dua{\opA^0(t_*^0+r^0)}{r^0} + C_{\mathrm{SM}}^0\dua{t_*^\perp+r^\perp-t_*^\perp}{r^\perp}_\VV\\
&=\dua{\opA^0(t_*^0+r^0)-\opA^0(t_*^0)}{r^0} + \dua{\opA^0(t_*^0)-\opA^1(t_*^1)}{r^0} + C_{\mathrm{SM}}^0\|r^\perp\|_\VV^2\\
&\ge C_{\mathrm{SM}}^0 \|r^0\|_\VV^2 -L^0\kappa \|r^0\|_\VV  + C_{\mathrm{SM}}^0\|r^\perp\|_\VV^2 \\
&\ge  (C_{\mathrm{SM}}^0 \|r^1\|_\VV- L^0\kappa )\|r^1\|_\VV=(C_{\mathrm{SM}}^0 \delta- L^0\kappa )\delta >0,
\end{align*}
where we in the last step used that $\|r^1\|_\VV^2 = \|r^0\|_\VV^2 + \|r^\perp\|_\VV^2$.
Furthermore,
\begin{align*}
A_1&=\dua{\opA^1(t^1_*+r^1)}{r^1}=\dua{\opA^1(t^1_*+r^1)-\opA^1(t^1_*)}{r^1}\ge  C_{\mathrm{SM}}^1 \|r^1\|_\VV^2>0.
\end{align*}

Using \cref{smindex} and the uniqueness of the zero $t_*^1$ in $B_\VV(t_*^1,\delta)$,
$$
\deg(\opK_\LI(\cdot,1),B_\VV(t_*^1,\delta),0)=\deg(\opA^1,B_\VV(t_*^1,\delta),0)=1.
$$
Since we have already seen that $\opK_\LI(t^1,\lambda)\neq 0$ for all $t^1\in\partial B_\VV(t_*^1,\delta)$ and $\lambda\in[0,1]$,
the homotopy invariance of the degree can be applied to get
$$
\deg(\opK_\LI(\cdot,0),B_\VV(t_*^1,\delta),0)=1.
$$
Using the existence property of the degree \cref{degcoro} (ii), there exists $t^0_{**}\in\VV^0$ such that $(t^0_{**},t_*^\perp)\in B_{\VV^1}(t_*^1,\delta)$ and
$\dua{\opK_\LI(t_{**}^0,0)}{s^0}=\dua{\opA^0(t_{**}^0)}{s^0}=0$ for all $s^0\in\VV^0$, which is what we wanted to prove.
Uniqueness follows form the local strong monotonicity of $\opA^0$.
\end{proof}

In the special case when $\opA^0$ is given as a projection of $\opA^1$ (see \cref{smproj}), we obtain the following.

\begin{corollary}
Suppose that $t_*^1\in\VV^1$ is a zero of $\opA^1$. Let $\kappa=\|t_*^\perp\|_\VV$ and suppose the following hold true.
\begin{enumerate}[label=\normalfont(\roman*)]
\item $L^0 \kappa=\sup_{\|u^0\|_\VV=1} \dua{\opA^1(t_*^0)-\opA^1(t_*^1)}{u^0}$ for some $L^0 < \infty$. 
\item $\opA^1$ is strongly monotone in $B_{\VV^1}(t_*^1,\delta)$ with constant $C_{\mathrm{SM}}>0$.
\item 
$\kappa< \delta\frac{C_{\mathrm{SM}}}{C_{\mathrm{SM}}+L^0}$.
\end{enumerate}
Then there exists a unique $t_{**}^0\in B_{\VV^0}(t_*^1,\delta-\kappa)$ such that $\opA^1(t_{**}^0)=0$. % and $i(\opA^0,t_{**}^0)=1$.
\end{corollary}
\begin{proof}
It is enough to recall that $\opA^0=\Pi_{\VV^0}\opA^1|_{\VV^0}$ is strongly monotone with constant $C_{\mathrm{SM}}$ in 
$$
B_{\VV^0}(t_*^0, \sqrt{\delta^2-\kappa^2})\supset B_{\VV^0}(t_*^0, \delta-\kappa).
$$
%and the result follows by 
Applying \cref{smexi} with $C_{\mathrm{SM}}^0 = C_{\mathrm{SM}}$ and $\delta_0 = \delta -\kappa$ gives the result.
\end{proof}

Condition (i) certainly holds, since the CC mapping is locally Lipschitz (see \cite{schneider2009analysis,rohwedder2013error}), but the constant $L^0$ here is possibly smaller than the usual Lipschitz constant (which, in turn, is bounded from below by the strong monotonicity constant).

\subsection{Kowalski--Piecuch homotopy}\label{seckp}

In this section, we consider another homotopy that can be used to analyze the connection between the solutions to CC methods of different rank-truncation levels. 
The approach was pioneered by the chemists K.~Kowalski and P.~Piecuch \cite{piecuch2000search}, who conducted a comprehensive numerical study based on 
this idea.\footnote{They call it the ``$\beta$--nested equations'', $\beta$ being the homotopy parameter.} They considered complex amplitudes and
the following discussion is easily extended to that case.

\vspace{1em}
\noindent\textbf{Assumption.} 
Let $\VV^1=\VV^0\oplus\VV^\angle$ be an \emph{$\ell^2$-orthogonal} direct sum decomposition
 of a real amplitude space $\VV^1=\VV(G^1)$, where $\VV^0$ is an amplitude space corresponding 
to some lower truncation level. More precisely, assume that there is a rank $\rho \geq 1$, such that $\VV^0$ contains all amplitudes with rank $\le\rho$ and 
$\VV^\angle:=(\VV^0)^{\perp_{\ell^2}}$ contains 
all amplitudes with rank $>\rho$. In the notations introduced in Part I,
$$
\VV^0=\VV(G^0)=\VV(G(1,\ldots,\rho))\cap \VV^1, \,\,\,\text{and}\,\,\,  \VV^\angle=\VV(G^\perp)=\VV(G(\rho+1,\ldots,N))\cap \VV^1,
$$
where $G^0=G(1,\ldots,\rho)\cap G^1$ and $G^\perp=G(\rho+1,\ldots,N)\cap G^1$. Hence, any $t^1\in\VV^1$ may be uniquely decomposed as $t^1=t^0+t^\angle$,
where $t^0\in\VV^0$, $t^\angle\in\VV^\angle$ and $\dua{t^0}{t^\angle}_{\ell^2}=0$. %\newline\vspace{1em}

\vspace{1em} \noindent Let $\VV=\VV(G^\full)$ be the full amplitude space and suppose that $\opA : \VV^1\to (\VV^1)^*$ is the SRCC mapping \cref{FCC}.
Write
\begin{equation}\label{kpcalc}
\begin{aligned}
    \vdua{\opA&(t^1)}{s^1} = \dua{e^{-T^1} \ham_K e^{T^1} \Phi_0}{S^0\Phi_0} + \dua{e^{-T^1} \ham_K e^{T^1} \Phi_0}{S^\angle\Phi_0}\\
	&=\dua{(e^{-T^0}+e^{-T^0}(e^{-T^\angle}-I)) \ham_K (e^{T^0}+e^{T^0}(e^{T^\angle}-I)) \Phi_0}{S^0\Phi_0} + \dua{e^{-T^1} \ham_K e^{T^1} \Phi_0}{S^\angle\Phi_0}\\
	&= \vdua{\opA^0(t^0)}{s^0} + \dua{e^{-T^0}(e^{-T^\angle} - I) \ham_K e^{T^0}(e^{T^\angle} - I) \Phi_0}{S^0\Phi_0 } \\
     &\quad+ \dua{e^{-T^0}(e^{-T^\angle} - I) \ham_K e^{T^0} \Phi_0}{S^0 \Phi_0} + \dua{e^{-T^0} \ham_K e^{T^0}(e^{T^\angle} - I ) \Phi_0}{S^0\Phi_0} + \dua{e^{-T^1} \ham_K e^{T^1}\Phi_0}{S^\angle\Phi_0} \\
 	   	&= \vdua{\opA(t^0)}{s^0} +  \dua{e^{-T^0} \ham_K e^{T^0}(e^{T^\angle} - I ) \Phi_0}{S^0\Phi_0} + \vdua{\opA(t^1)}{s^\angle},
\end{aligned}
\end{equation}
where in the last step the second and the third terms of the penultimate expression vanish because
$$
(e^{-T^\angle}-I)^\dag S^0\Phi_0=-(T^\angle)^\dag S^0\Phi_0-\frac{1}{2}\left((T^\angle)^\dag\right)^2 S^0\Phi_0 -\ldots=0,
$$
due to the definition of the spaces $\VV^0$ and $\VV^\angle$. Note that this last relation would not hold in the case $\VV^0=\VV(G(\mathrm{D}))\cap \VV^1$, $\VV^1=\VV(G^\full)$,
which is actually excluded by assumption.

Motivated by the preceding calculation in \cref{kpcalc}, we define the \emph{Kowalski--Piecuch homotopy} $\opK_{\KP} : \VV^1\times[0,1]\to(\VV^1)^*$ via
the instruction
\begin{equation}\label{pphomo}
\begin{aligned}
    \vdua{\opK_{\KP}(t^1,\lambda)}{s^1}&:=\dua{ \ham_K(t^0) \Phi_0}{ S^0\Phi_0} + \dua{\ham_K(t^1) \Phi_0}{ S^\angle\Phi_0}+ \lambda \dua{\ham_K(t^0)(e^{T^\angle} - I ) \Phi_0}{S^0\Phi_0}
\end{aligned}
\end{equation}
for all $t^1,s^1\in \VV^1$ and $\lambda\in[0,1]$.
Here, the relation $\opK_{\KP}(t^1_*,\lambda)=0$ is the same as the system \cite[eqs. (90)--(91) and (93)]{piecuch2000search}. 

It is obvious that $\opK_\KP$ is an admissible homotopy. 
However, it is unclear whether it is faithful or not. In fact, $\opK_{\KP}(t^1_{**},0)=0$ is equivalent to the system 
\begin{align}\label{kpsys0}
\dua{\ham_K(t_{**}^0)\Phi_0}{S^0\Phi_0}&=0,\\\label{kpsys1}
\dua{\ham_K(t_{**}^0+t_{**}^\angle)\Phi_0}{S^\angle\Phi_0}&=0,
\end{align}
for all $s^1\in\VV^1$. In other words, the usual SRCC equation $\opA(t_{**}^0)=0$, \cref{kpsys0}, is augmented with an additional equation for $t_{**}^\angle$, \cref{kpsys1},
which, in turn, depends on $t_{**}^0$.
It is not obvious at all that \cref{kpsys1} has a solution $t_{**}^\angle$ for a given zero $t_{**}^0$. The extensive numerical evidence in \cite{piecuch2000search} clearly indicates that \cref{kpsys1} admits a solution in various circumstances. 

Before stating our existence result we recast the KP homotopy into a more convenient form, which we already encountered in \cref{simphom}.
\begin{lemma}\label{KPrw}
The following formula holds true:
\begin{equation}\label{kpcomp}
\vdua{\opK_{\KP}(t^1,\lambda)}{s^1}=\vdua{\opA(t^0+\lambda t^\angle)}{s^0} + \vdua{\opA(t^1)}{s^\angle},
\end{equation}
for all $t^1,s^1\in \VV^1$ and $\lambda\in[0,1]$.
\end{lemma}
\begin{proof}
It is enough to prove that
\begin{equation}\label{kphomoproof}
\vdua{\opK_{\KP}(t^1,\lambda)}{s^1}=\dua{e^{-T^0} \ham_K e^{T^0+\lambda T^\angle}\Phi_0}{S^0\Phi_0}
+ \dua{e^{-T^1} \ham_K e^{T^1} \Phi_0}{ S^\angle\Phi_0},
\end{equation}
because $(e^{-\lambda T^\angle})^\dag S^0\Phi_0=S^0\Phi_0$.
To see \cref{kphomoproof}, note that in the expansion $e^{T^0+\lambda T^\angle}=e^{T^0}(I+\lambda T^\angle + \frac{\lambda^2}{2!}(T^\angle)^2 + \ldots+\frac{\lambda^N}{N!}(T^\angle)^N)$
the quadratic and higher-order terms do not contribute because the excitation rank of $\ham_K (T^\angle)^k\Phi_0$ exceeds $\rho$ for $k=2,\ldots,N$,
due to the fact that $\ham_K$ is a two-body operator.
\end{proof}

To formulate the existence result, suppose that $t_*^1\in\VV^1$ is a non-degenerate zero of $\opA$.
Since $\VV^1$ is assumed to be finite-dimensional, it is always possible to choose an $\alpha>0$ so that
\begin{equation}\label{alphacoerc}
\vdua{(\opA'(t_*^1) + \alpha I)r^1}{r^1}\ge \gamma_\alpha \|r^1\|^2_\VV\quad\text{for all}\quad r^1\in\VV^1,
\end{equation}
for some $\gamma_\alpha>0$. This is also true in the complex 
case with a ``$\Real$'' added to the left-hand side.

Define the operator $\Theta_\alpha:\VV\to\VV$ via\footnote{The authors learned this trick from \cite{buffa2003boundary,buffa2005remarks}.}
\begin{equation}\label{thetadef}
\vdua{\opA'(t_*^1)u^1}{\Theta_\alpha v^1}=\vdua{(\opA'(t_*^1)+\alpha I)u^1}{v^1} \quad\text{for all}\quad u^1,v^1\in\VV^1.
\end{equation}
Then $\Theta_\alpha$ is well-defined, as long as $\ker \opA'(t_*^1)=\{0\}$, i.e. that $t_*^1$ is non-degenerate.

\begin{theorem}[Existence for KP]\label{kpexi}
Let $t_*^1\in\VV^1$ be a non-degenerate zero of $\opA$.
Suppose the following.
\begin{enumerate}[label=\normalfont(\roman*)]
\item With $\theta_0=\|\Pi_0 (\Theta_\alpha-I) \Pi_0\|_{\mathcal{L}(\VV)}^2$ and $\theta_\angle=\|\Pi_0 (\Theta_\alpha-I) \Pi_\angle\|_{\mathcal{L}(\VV)}^2$,
there holds
$$
\eta:=(1-g)\frac{\gamma_\alpha-\frac{1}{2}M_\delta\|\Theta_\alpha\|_{\mathcal{L}(\VV)}\delta}{\Delta(t_*^1)+M_\delta\delta}-\frac{1}{2}\max\{ \epsilon + 2(1+\epsilon^{-1})\theta_0, 2(1+\epsilon^{-1})\theta_\angle \}>\frac{1}{2},
$$
for some $\epsilon>0$ and $\delta>0$, where $M_\delta$ was defined in \cref{mdelta}. Here, 
$$
\Delta(t_*^1)=\max_{\substack{\|u^0\|_\VV=1\\ \|v^\angle\|_\VV=1}} |\dua{(\ham_K + [(T_*^1)_1^\dag,\ham_K]) U^0\Phi_0}{V^\angle\Phi_0}|.
$$
 Also, $\alpha>0$ and $\gamma_\alpha>0$ satisfy \cref{alphacoerc} and
 $0<g<1$ is such that $|\dua{t^0}{t^\angle}_\VV|\le g \|t^0\|_\VV \|t^\angle\|_\VV$ for all $t^1\in\VV^1$.
\item With $\kappa=\|t_*^\angle\|_\VV$, there holds
\begin{equation}\label{kappacond2}
\kappa< \frac{2\sqrt{\eta}-\sqrt{2}}{2-\sqrt{2}+2\sqrt{\eta}} \delta.
\end{equation}
\end{enumerate}
Let $D=\{ t_*^1 + r^1 \in \VV^1 : \|r^0\|_\VV^2 + \|r^\angle\|_\VV^2 < \frac{1}{2}(\delta-\kappa)^2 \}$.
Then, for any $\lambda\in[0,1)$, there exists $t_{**}^1(\lambda)\in D$ such that $\opK_\KP(t_{**}^1(\lambda),\lambda)=0$. Furthermore, $\deg(\opK_\KP(\cdot,\lambda),D,0)\equiv d\neq 0$
for all $\lambda\in[0,1]$. In particular, there exists $t_{**}^1\in D$ such that $\opA(t_{**}^0)=0$.
\end{theorem}
For the proof, see \cref{kpexapp}.

\begin{remark}
\leavevmode
\begin{enumerate}[label=\normalfont(\roman*)]
\item A crucial difference between the linear homotopy and the KP homotopy is that while the decomposition used for $\opK_\LI$
is $\VV$-orthogonal, the decomposition for $\opK_\KP$ is $\ell^2$-orthogonal---the computation \cref{kpcalc} and \cref{KPrw} exploit this heavily. 
Nevertheless, we used the $\VV$-inner product in the existence result for $\opK_\KP$. This geometric discrepancy is reflected in condition (i) above. 
\item Using the Cauchy--Schwarz inequality it is clear that $|\dua{t^0}{t^\angle}_\VV|<\|t^0\|_\VV \|t^\angle\|_\VV$ for all $t^1\in\VV^1$,
due to the fact that $\VV^1=\VV^0\oplus\VV^\angle$. The maximum of the function $t^1\mapsto \frac{\dua{t^0}{t^\angle}_\VV}{\|t^0\|_\VV \|t^\angle\|_\VV}$ is attained on
$\|t^0\|_\VV=1$, $\|t^\angle\|_\VV=1$ and this maximum may be taken as the $0<g<1$ of condition (i).
\item In the coercive case, i.e. when \cref{alphacoerc} holds with $\alpha=0$, we have $\Theta_0=I$, so that $\theta_0=\theta_\angle=0$. Letting $\epsilon\to 0$, condition (i) simplifies to
$$
\eta=\frac{(1-g)(\gamma_0-\frac{1}{2}M_\delta \delta)}{\Delta(t_*^1)+M_\delta \delta}>\frac{1}{2}.
$$
This last condition in turn reduces to $\gamma_0 > M_\delta \delta$ as $g$ and $\Delta(t_*)$ approaches zero.  
\item It is interesting to note that while the linear homotopy $\opK_{\mathrm{L}}$ involves the ``targeted'' solution $t_*^1$, the KP homotopy $\opK_\KP$ does not. In this sense,
$\opK_\KP$ is ``universal''. 
\item The result is straightforward to extend to the complex case.
\item A careful inspection of the proof shows that the result can be generalized to the case when $\VV^1=\VV^0\oplus\VV^\angle$ is an arbitrary $\ell^2$-orthogonal
direct sum decomposition as long as the homotopy $\opK_\KP$ is \emph{defined} via \cref{kpcomp} (cf. \cref{simphom}). In this more general case, $\Delta(t_*^1)$ is
given by a more complicated expression.
\end{enumerate}
\end{remark}

The constant $\Delta(t_*^1)$ also deserves some explanation. Roughly speaking, the ``defect'' $\Delta(t_*^1)$ measures how much the subspace $\VC^0\subset\VC$ deviates from being an
invariant subspace of the operator $(\ham_K + [(T_*^1)_1^\dag,\ham_K])_{\VC} : \VC\to\VC$. In addition, we can invoke \cref{fockcomm} 
to write 
$$
\Delta(t_*^1)=\max_{\substack{\|u^0\|_\VV=1\\\|v^\angle\|_\VV=1}} |\dua{(\mathcal{W}_K + [(T_*^1)_1^\dag,\mathcal{W}_K])U^0\Phi_0}{V^\angle\Phi_0}|.
$$
Said differently, $\Delta(t_*^1)=\|(\mathcal{W}_K + [(T_*^1)_1^\dag,\mathcal{W}_K])_{\VC}\|_{\mathcal{L}(\VC^0,\VC^\perp)}$.
Note that the term $[(T_*^1)_1^\dag,\mathcal{W}_K]$ can be eliminated via orbital
rotations (since $(t_*^1)_1=0$ can be achieved according to the Thouless theorem), as it involves single excitations only, in which case the amplitude dependence of $\Delta$ is removed. 
Hence, $\Delta$ quantifies how much the operator $(\mathcal{W}_K)_{\VC}$ leaves $\VC^0$ invariant. 

We summarize the above existence result in the corollary below that holds under the following structural assumptions.

\vspace{1em}
\noindent\textbf{Assumption (KPA).} $\Delta(t_*^1)$ can be made sufficiently small by an appropriate choice of the orbital basis and truncation level $1\le\rho<N$.

\vspace{1em}
\noindent\textbf{Assumption (KPB).} There is a $\delta_0>0$ such that $M_{\delta_0}$ is sufficiently small. 
\vspace{1em}

\begin{corollary}
Suppose that $t_*^1\in\VV^1$ is a non-degenerate zero of $\opA$ and that Assumptions (KPA) and (KPB) hold. 
If $\|t_*^\angle\|_\VV$ is sufficiently small, then there exists $t_{**}^1\in\VV^1$ in a neigborhood of $t_{*}^1$ such that $\opK_\KP(t_{**}^1,0)=0$.
\end{corollary}

\begin{remark}
Recall that $M_\delta$ only involves the second derivative $\opA''$ near $t_*^1$ (see \cref{mdelta}), so that (KPB) can be viewed as a ``perturbative''
assumption. Further, as we noted in \cref{Mdeltarmk}, $M_\delta$ involves the mapping $\zeta\mapsto[[\mathcal{W}_K(t_*^1+\zeta),\cdot],\cdot]\Phi_0$. 
Therefore, (KPB) may be viewed as the higher-order generalization of assumption on the smallness of the local Lipschitz constant of the mapping $\zeta\mapsto\mathcal{W}_K(t_*^1+\zeta)\Phi_0$
in \cite[Assumption BII]{rohwedder2013error} (see also \cref{ccenermk} (v)).	
\end{remark}

In the last step of the \emph{proof} of \cref{kpexi}, we could have invoked \cref{admcont} instead to obtain the existence of 
a connected component $\mathcal{C}$ of the zero set $\mathcal{Z}(\opK_\KP)$, such that $\mathcal{C}\cap (\mathcal{Z}(\opK_\KP))_j\neq\emptyset$ for $j=0,1$.
Under the above assumptions, this provides a theoretical basis for the ``solution trajectories'' observed in \cite{piecuch2000search}.
Further, combining \cref{kpexi} with \cref{indexrel}, we get for $t_{**}^1=t_{**}^1(0)$,
$$
\sum_{ t_{**}^1\in \mathcal{C}\cap (\mathcal{Z}(\opK_\KP))_0 } i(\opK_{\KP}(\cdot,0),t_{**}^1) = i(\opA, t_*^1).
$$
Since we do not have uniqueness for $t_{**}^1$ in this case, it is possible that the left-hand side may contain multiple terms which sum up to $i(\opA, t_*^1)$.
In particular, $i(\opK_{\KP}(\cdot,0),t_{**}^1)$ does not need to be $i(\opA, t_*^1)$.

To close this section, we calculate
the topological index for a non-degenerate zero at the $\lambda=0$ endpoint of the KP homotopy. Note that the index at $\lambda=1$ is 
simply given by \cref{indxformula} and \cref{indxformuladege}.

Fix $t^1\in\VV^1$ and define the operator $\wh{\ham_K}(t^1)$ (not to be confused with \cref{hamhats} in a different context),
\begin{equation}\label{kpmodham}
\widehat{\ham_K}(t^1)=\ham_K(t^1)-\sum_{\alpha\in\Xi(G^0)} \dua{\ham_K(t^1)\Phi_0}{\Phi_\alpha} X_\alpha.
\end{equation}
Notice that $\dua{\widehat{\ham_K}(t^1)\Phi_0}{S^0\Phi_0}=0$ for all $s^0\in\VV^0$.
Define the linear mapping $\wh{\ham_K}(t^1)_{\VC^0,\VC^\angle}:\VC^0\to\VC^\angle$ via
$\dua{\wh{\ham_K}(t^1)_{\VC^0,\VC^\angle}\Psi}{\Psi'}=\dua{\wh{\ham_K}(t^1)\Psi}{\Psi'}$
for all $\Psi\in \VC^0$ and $\Psi'\in\VC^\angle$.
We first calculate the derivative of $\opK_\KP(\cdot,0)$.
\begin{lemma}\label{kpderlem}
The derivative $\partial_1\opK_{\KP}(t_{**}^1,0):\VV^1\to (\VV^1)^*$ of $\opK_\KP(\cdot,0)$ at a zero $t_{**}^1$ is given by
\begin{align*}
&\dua{\partial_1\opK_{\KP}(t^1_*,0)u^1}{v^1}=\bdua{\begin{pmatrix}
\ham_K(t_{**}^0)_{\VC^0}  - \ene_{\cc}(t_{**}^0)  & 0 \\
\wh{\ham_K}(t_{**}^1)_{\VC^0,\VC^\angle}  &  \widehat{\ham_K}(t_{**}^1)_{\VC^\angle} - \ene_\cc(t_{**}^1)
\end{pmatrix}\begin{pmatrix}
U^0\Phi_0\\
U^\angle\Phi_0
\end{pmatrix}}{\begin{pmatrix}
V^0\Phi_0\\
V^\angle\Phi_0
\end{pmatrix}}
\end{align*}
for all $u^1,v^1\in \VV^1$.
\end{lemma}
\begin{proof}
A calculation analogous to the one in the proof of \cref{Adifflemm} shows that $D(t^1):= \partial_1\opK_{\KP}(t^1,0) : \VV^1\to (\VV^1)^*$ is given by
\begin{equation}\label{KPder}
\begin{aligned}
    \dua{D(t^1) u^1}{v^1} &= \dua{[\ham_K(t^0),U^0]\Phi_0}{V^0\Phi_0}
    + \dua{[\ham_K(t^1),U^1] \Phi_0}{ V^\angle\Phi_0} \\
    &=: D_1(t^1) + D_2(t^1)
\end{aligned}
\end{equation}
for all $t^1,u^1,v^1\in\VV^1$. %Here the last equality in \cref{KPder} defines $D_1$ and $D_2$. 
We set $t^1=t^1_{**}$ and evaluate $D_1$ and $D_2$. Write 
$$
(U^0)^\dag V^0\Phi_0 = \sum_{\alpha\in\Xi(G^1)\cup\{0\}} \dua{U^0\Phi_\alpha}{V^0\Phi_0} \Phi_\alpha,
$$
from which,
\begin{align*}
\dua{U^0\ham_K(t_{**}^0)\Phi_0}{V^0\Phi_0}&=\dua{\ham_K(t_{**}^0)\Phi_0}{(U^0)^\dag V^0\Phi_0}\\
&=\Bigg(\sum_{\alpha\in\Xi(G^0)}+\sum_{\alpha\in\Xi(G^0)^c}\Bigg) \dua{\ham_K(t_{**}^0)\Phi_0}{\Phi_\alpha} \dua{ U^0 \Phi_\alpha}{V^0\Phi_0} + \ene_\cc(t_{**}^0)\dua{U^0 \Phi_0}{V^0\Phi_0}.
\end{align*}
Here, the first sum vanishes because of \cref{kpsys0} and the second by the orthogonality of $\VV^0$ and $\VV^\angle$.
Consequently, 
\begin{equation*}
    D_1(t_{**}^1) = \dua{(\ham_K(t_{**}^0) - \ene_\cc(t_{**}^0) )U^0\Phi_0}{V^0\Phi_0} . 
\end{equation*}
Analogously, \cref{kpsys1} implies that
\begin{align*}
\dua{U^\angle\ham_K(t_{**}^1)\Phi_0}{V^\angle\Phi_0}&=\sum_{\alpha\in\Xi(G^0)} \dua{\ham_K(t_{**}^1)\Phi_0}{\Phi_\alpha} \dua{X_\alpha U^\angle \Phi_0}{V^\angle\Phi_0} + \ene_\cc(t_{**}^1)\dua{U^\angle\Phi_0}{V^\angle\Phi_0},
\end{align*}
and
$$
\dua{U^0\ham_K(t_{**}^1)\Phi_0}{V^\angle\Phi_0}=\sum_{\alpha\in\Xi(G^0)} \dua{\ham_K(t_{**}^1)\Phi_0}{\Phi_\alpha} \dua{X_\alpha U^0 \Phi_0}{V^\angle\Phi_0}.
$$
Thus, 
%Substituting these into \cref{KPder} evaluated at $t^1=t^1_*$, we get
\begin{align*}
D_2(t_{**}^1) = \dua{ (\widehat{\ham_K}(t_{**}^1) -\ene_\cc(t_{**}^1))U^\angle\Phi_0}{V^\angle\Phi_0}
+ \dua{\wh{\ham_K}(t_{**}^1)U^0\Phi_0}{ V^\angle\Phi_0},
\end{align*}
and the stated expression now follows.
\end{proof}

The preceding Lemma implies the index formula for the KP homotopy.

\begin{theorem}[Index formula for KP -- non-degenerate case]\label{kpindex}
Suppose that $t_{**}^1\in\VV^1$ is a zero of $\opK_\KP(\cdot,0)$. Then $t_{**}^1$ is non-degenerate if and only if the conditions
\begin{enumerate}[label=\normalfont(\Roman*)]
\item $\ene_\cc(t_{**}^0)\not\in\sigma(\ham_K(t_{**}^0)_{\VC^0})$,
\item $\ene_\cc(t_{**}^1)\not\in\sigma(\widehat{\ham_K}(t_{**}^1)_{\VC^\angle})$
\end{enumerate}
both hold true, and in this case the topological index of $\opK_\KP(\cdot,0)$ at $t_{**}^1$ is given by
$i(\opK_{\KP}(\cdot,0),t_{**}^1)=(-1)^{\nu^0+\nu^\angle}$,
where 
\begin{align*}
\nu^0&=|\{j: \ene_j(\ham_K(t_{**}^0)_{\VC^0})\in\RR, \; \ene_j(\ham_K(t_{**}^0)_{\VC^0})<\ene_{\cc}(t_{**}^0)\}|,\\
\nu^\angle&=|\{j: \ene_j(\widehat{\ham_K}(t_{**}^1)_{\VC^\angle})\in\RR, \; \ene_j(\widehat{\ham_K}(t_{**}^1)_{\VC^\angle})<\ene_{\cc}(t_{**}^1)\}|.
\end{align*}
\end{theorem}
\begin{proof}
The proof follows from \cref{kpderlem} along similar lines as \cref{indxformula},
\begin{align*}
i(\opK_{\KP}(\cdot,0),t_{**}^1)&=\sgn \prod_{j\ge 0} (\ene_j(\ham_K(t_{**}^0)_{\VC^0}) - \ene_{\cc}(t_{**}^0))  \sgn \prod_{j\ge 0} (\ene_j(\widehat{\ham_K}(t_{**}^1)_{\VC^\angle}) - \ene_\cc(t_{**}^1)).
\end{align*}
\end{proof}

\subsection{An energy error estimate}\label{enesec}

In this this section we derive an energy error estimate for general eigenstates for the KP homotopy using the results of \cref{kpapp}.

\begin{theorem}[Energy error estimate]\label{kpenethm}
Let $\VV^1=\VV(G^\full)$ and suppose that $t_*^1\in\VV^1$ is a zero of $\opA$, and that $t_{**}^1\in\VV^1$ is a zero of $\opK_\KP(\cdot,0)=0$.
If the nonorthogonality condition $\dua{e^{T^0_{**}} \Phi_0}{e^{T_*^1}\Phi_0}\neq 0$ holds true, then
\begin{equation}\label{kpeneest}
|\ene_{\cc}(t_{**}^1)-\ene_{\cc}(t_*^1)|\le C(t_{**}^1,t_{*}^1) \|t_{**}^\angle\|_\VV,
\end{equation}
where
$$
C(t_{**}^1,t_{*}^1)=(C^2+M(t_{**}^1))\frac{\|\Pi_{\VC^\angle}(e^{T^0_{**}})^\dag  \Pi_{\VC^\angle}e^{T_*^1}\Phi_0\|_{\HC^1}}{|\dua{e^{T^0_{**}} \Phi_0}{e^{T_*^1}\Phi_0}|}
$$
is bounded as $\|t_{**}^\angle\|_\VV\to 0$. 
Here, $C$ is the norm equivalence constant from \cref{focknorm} and $M(t_{**}^1)=\max_{\xi\in[t_{**}^0,t_{**}^1]} \|u^1\mapsto \Pi_{\VC^\angle}[\mathcal{W}_K(\xi),U^1]\|_{\mathcal{L}(\VV,\HC^{-1})}$.
\end{theorem}
\begin{proof}
Setting $\Psi=e^{T_*^1}\Phi_0$ and $\lambda=0$ in \cref{kpthm} (II), we have
\begin{equation}\label{kpineqproof}
|\ene_\cc(t_{**}^1)-\ene_{\cc}(t_*^1)|=\frac{|\dua{(\ham_K(t^1_{**}) - \ham_K (t^0_{**}))\Phi_0}{ \Pi_{\VC^\angle}(e^{T^0_{**}})^\dag  \Pi_{\VC^\angle}e^{T_*^1}\Phi_0}|}{|\dua{e^{T^0_{**}} \Phi_0}{e^{T_*^1}\Phi_0}|}.
\end{equation}
Note that using \cref{fockcomm} we may write
\begin{align*}
(\ham_K(t^1_{**}) - \ham_K (t^0_{**}))\Phi_0&= [\mathcal{F}_K,T_{**}^\angle]\Phi_0 +  (\mathcal{W}_K(t^1_{**}) - \mathcal{W}_K(t^0_{**}))\Phi_0 \\
&= \sum_{\gamma\in\Xi(G^\angle)} \epsilon_\gamma (t_{**}^\angle)_\gamma\Phi_\gamma +  (\mathcal{W}_K(t^1_{**}) - \mathcal{W}_K(t^0_{**}))\Phi_0.
\end{align*}
Hence, we can bound the numerator of the right-hand side of \cref{kpineqproof} as
\begin{align*}
&\Big|\sum_{\gamma\in\Xi(G^\angle)} \epsilon_\gamma (t_{**}^\angle)_\gamma \dua{\Phi_\gamma}{\Pi_{\VC^\angle}(e^{T^0_{**}})^\dag  \Pi_{\VC^\angle}e^{T_*^1}\Phi_0}\Big|+ \Big|\dua{(\mathcal{W}_K(t^1_{**}) - \mathcal{W}_K(t^0_{**}))\Phi_0}{\Pi_{\VC^\angle}(e^{T^0_{**}})^\dag  \Pi_{\VC^\angle}e^{T_*^1}\Phi_0}\Big| .
\end{align*}
Using the Cauchy--Schwarz inequality, the first term may be further bounded as
\begin{align*}
\sum_{\gamma\in\Xi(G^\angle)} \epsilon_\gamma (t_{**}^\angle)_\gamma \dua{\Phi_\gamma}{\Pi_{\VC^\angle}(e^{T^0_{**}})^\dag  \Pi_{\VC^\angle}e^{T_*^1}\Phi_0}\le \vertiii{t_{**}^\angle} \vertiii{\Pi_{\VC^\angle}(e^{T^0_{**}})^\dag  \Pi_{\VC^\angle}e^{T_*^1}\Phi_0},
\end{align*}
where the Fock norm $\vertiii{\cdot}$ was defined in \cref{focknorm}.
For the second term, we use the intermediate value inequality to obtain the bound
$M(t_{**}^1)\|t_{**}^\angle\|_\VV \|\Pi_{\VC^\angle}(e^{T^0_{**}})^\dag  \Pi_{\VC^\angle}e^{T_*^1}\Phi_0\|_{\HC^1}$.
\end{proof}
\begin{remark}\label{ccenermk}
\leavevmode
\begin{enumerate}[label=\normalfont(\roman*)]
\item If the nonorthogonality condition holds and $t_{**}^\angle=0$, then according to \cref{kpsys0}-\cref{kpsys1}, $t_{**}^0$ is an FCC solution such that 
$\dua{e^{T^0_{**}} \Phi_0}{e^{T_*^1}\Phi_0}\neq 0$.
In this case, \cref{kpeneest} implies that the energy error is zero: $\ene_\cc(t_{**}^0)=\ene_{\cc}(t_*^1)$.
\item In the practically relevant case $\rho\ge 2$, using \cref{kpenesd}, we have that $\ene_{\cc}(t_{**}^1)=\ene_{\cc}(t_{**}^0)$ so the left-hand side of \cref{kpeneest}
does not involve $t_{**}^\angle$ at all. 
\item The appearance of the quantity $\|t_{**}^\angle\|_\VV$ allows us to view the auxiliary equation \cref{kpsys1} as providing an \emph{a posteriori} error estimate. 
\item If the nonorthogonality condition does not hold, i.e. $\dua{e^{T^0_{**}} \Phi_0}{e^{T_*^1}\Phi_0}=0$, then 
$e^{T^0_{**}} \Phi_0$ and $e^{T_*^1}\Phi_0$ represent different eigenstates if $t_*^1$ is assumed to be non-degenerate.
While $e^{T^0_{**}} \Phi_0$ itself does not satisfy the Schr\"odinger equation, it must be viewed as an approximation to an eigenstate \emph{different} from $e^{T_*^1}\Phi_0$.
\item Note that a local Lipschitz assumption (on a ball including $t_{**}^1$ and $t_{**}^0$) with constant $L$ on the mapping $t\mapsto \mathcal{W}_K(t) \Phi_0$ can be used to obtain the result of the theorem with $M$ replaced by $L$. Such an assumption is akin to Assumption B.II in \cite{rohwedder2013error} where it is used for guaranteeing the local strong monotonicity of the CC mapping. 
\end{enumerate}
\end{remark}

%%%%%%%%%%%%%%%%%%%%%%%%%%%%%%%%%%%%%%%%%%%%%%%%%%%%%%%%%%%%%%%%%%%%%%%%%%%%%%%%%%%%%%%%%%%%%%%%%%%
%% CONCLUSIONS
%%%%%%%%%%%%%%%%%%%%%%%%%%%%%%%%%%%%%%%%%%%%%%%%%%%%%%%%%%%%%%%%%%%%%%%%%%%%%%%%%%%%%%%%%%%%%%%%%%%

\section{Conclusions and further work}

In this second part of a series of two articles, we analyzed the SRCC method. We provided background material for the setting of the quantum-mechanical problems
the SRCC method is aimed at in \cref{secbg}. Then, in \cref{sectopdeg}, we gave a brief summary of topological degree theory since this served as our main tool for the analysis.
 
The main discussion is contained in \cref{secsrcc}, where we began with the definition of the SRCC mapping and some elementary considerations in \cref{ccbasic}. 
We then considered the local properties of the SRCC mapping in \cref{secsrccop}.
It turned out that the topological index of the CC mapping (\cref{indxformula}) is connected with the nonvariational property of the CC method (\cref{srcceom}),
and the eigenvalues of the Fock operator (\cref{focknondeg}). In the degenerate case, the classic Leray reduction formula provided the topological index (\cref{indxformuladege}).
We also discussed the case when the cluster amplitudes are allowed to be complex in \cref{secsrcccmplx}.

In \cref{seccont}, we discussed how certain homotopies can be used to analyze the CC method, in particular to prove the existence of a truncated CC solution through the use
of topological degree theory. This was done using an idea well-known both in nonlinear analysis and in quantum chemistry: that an appropriate homotopy ``connects'' the truncated problem with the exact problem (essentially the Schr\"odinger equation), therefore one is able to infer (homotopy-invariant) information regarding the former problem from the latter.
As an introductory example, we considered the linear homotopy (\cref{smexi}).

Next, in \cref{seckp}, motivated by the works of the chemists Kowalski and Piecuch, we considered a homotopy that connects CC mappings corresponding to different truncation levels.
Using this, we proved an existence result for the said homotopy (\cref{kpexi}), which also implies the existence of a truncated CC solution under certain assumptions.
The index formula for the KP homotopy was also derived in the non-degenerate case (\cref{kpindex}). 
Using a known result about the KP homotopy (\cref{kpapp}), we also derived an energy error estimate in \cref{enesec}.

Finally, let us discuss some possible directions of research. Clearly, it would be interesting to extend our analysis to the infinite-dimensional case.
 An obvious next step would be the analysis of the JM-MRCC method (see Section 4.2 of Part I).
It would be also interesting to look at the Extended CC (ECC) \cite{Arponen1983,laestadius2018analysis} and the Unitary CC (UCC) methods \cite{bartlett1989alternative}.

%%%%%%%%%%%%%%%%%%%%%%%%%%%%%%%%%%%%%%%%%%%%%%%%%%%%%%%%%%%%%%%%%%%%%%%%%%%%%%%%%%%%%%%%%%%%%%%%%%%
%% APPENDICES
%%%%%%%%%%%%%%%%%%%%%%%%%%%%%%%%%%%%%%%%%%%%%%%%%%%%%%%%%%%%%%%%%%%%%%%%%%%%%%%%%%%%%%%%%%%%%%%%%%%

\appendix

\section{Proof of \cref{kpexi}}\label{kpexapp}

We first prove that $\opK_\KP(\cdot,\lambda)\neq 0$ on $\partial D$ for all $\lambda\in[0,1]$.
Set $t^1=t_*^1 + r^1$ and $s^1=\Theta_\alpha r^1$ in \cref{kpcomp} where $r^1\in\partial D$, and write
\begin{align*}
&\dua{\opK_\KP(t_*^1 + r^1,\lambda)}{\Theta_\alpha r^1}=\dua{\opA(t_*^0+r^0+\lambda t_*^\angle + \lambda r^\angle)}{\Pi_0\Theta_\alpha r^1} 
+ \dua{\opA(t_*^1+r^1)}{\Pi_\angle\Theta_\alpha r^1}\\
&=\dua{\opA(t_*^0+\lambda t_*^\angle + r^0 + \lambda r^\angle)-\opA(t_*^1+r^1)}{\Pi_0\Theta_\alpha r^1}+ \dua{\opA(t_*^1+r^1)}{\Theta_\alpha r^1}=:\mathrm{(I)}+\mathrm{(II)}.
\end{align*}
For (I), Taylor expansion around $t_*^1$ leads to
\begin{align*}
\mathrm{(I)}&=(\lambda-1) \dua{\opA'(t_*^1)(t_*^\angle + r^\angle)}{\Pi_0\Theta_\alpha r^1} + \dua{\mathcal{R}_2(t_*^1,(\lambda-1)t_*^\angle + r^0 +\lambda r^\angle)
-\mathcal{R}_2(t_*^1,r^1)}{\Pi_0\Theta_\alpha r^1}\\
&\ge (\lambda-1)\dua{\opA'(t_*^1)(t_*^\angle + r^\angle)}{\Pi_0\Theta_\alpha r^1} - M\|t_*^\angle + r^\angle\|\|\Pi_0\Theta_\alpha r^1\|\\
&\ge -(\Delta(t_*^1)+M) \|t_*^\angle + r^\angle\|_\VV\|\Pi_0\Theta_\alpha r^1\|_\VV,
\end{align*}
where we used the intermediate value inequality. Here, letting $\Psi^0\in\VC^0$ correspond to the amplitude $\Pi_0\Theta_\alpha r^1$, 
\begin{align*}
\dua{\opA'(t_*^1)(t_*^\angle+r^\angle)}{\Pi_0\Theta_\alpha r^1}&=\dua{(\ham_K(t_*^1)-\ene_\cc(t_*^1))(T_*^\angle+R^\angle)\Phi_0}{\Psi^0}=\dua{\ham_K(t_*^1) (T_*^\angle + R^\angle)\Phi_0}{\Psi^0}\\
&=\dua{(\ham_K + [\ham_K,T_*^1] + \ldots) (T_*^\angle + R^\angle)\Phi_0}{\Psi^0}\le \Delta(t_*^1) \|t_*^\angle + r^\angle\|_\VV \|\Pi_0\Theta_\alpha r^1\|_\VV,
\end{align*}
 where the first equality is \cref{SRCCDZ}, and in the last inequality we exploited that $\ham_K$ decreases the excitation rank at most by 2 (so that the single amplitudes $(t_*^1)_1$ of $t_*^1$ only contribute).
Also, the following estimate holds,
\begin{align*}
M&=\max_{\xi\in [(\lambda-1)t_*^\angle + r^0 +\lambda r^\angle,r^1]} \|\partial_2\mathcal{R}_2(t_*^1,\xi)\|_{\mathcal{L}(\VV,\VV^*)}\\
&=\max_{\xi\in [(\lambda-1)t_*^\angle + r^0 +\lambda r^\angle,r^1]} \|\opA'(t_*+\xi)-\opA'(t_*)\|_{\mathcal{L}(\VV,\VV^*)}\\
&\le \max_{\xi\in [(\lambda-1)t_*^\angle + r^0 +\lambda r^\angle,r^1]} \|\xi\|_\VV \max_{\zeta\in [0,\xi]}  \|\opA''(t_*^1+\zeta)\|_{\mathcal{L}(\VV\times\VV,\VV^*)} \\
&\le M_\delta \max_{\xi\in [(\lambda-1)t_*^\angle + r^0 +\lambda r^\angle,r^1]} \|\xi\|_\VV\le M_\delta( \|r^0\|_\VV + \|r^\angle\|_\VV + \kappa)\le M_\delta \delta.
\end{align*}

For (II), we have using \cref{alphacoerc}, \cref{thetadef} and Taylor's theorem,
\begin{align*}
\mathrm{(II)}&=\dua{\opA'(t_*^1)r^1}{\Theta_\alpha r^1} - \dua{\mathcal{R}_2(t_*^1,r^1)}{\Theta_\alpha r^1}\ge \gamma_\alpha  \|r^1\|_\VV^2 - \tfrac{1}{2}M_\delta\|r^1\|_\VV^2 \|\Theta_\alpha r^1\|_\VV\\
&\ge (\gamma_\alpha -\tfrac{1}{2}M_\delta\|\Theta_\alpha\|_{\mathcal{L}(\VV)}\|r^1\|_\VV) \|r^1\|_\VV^2\ge (\gamma_\alpha -\tfrac{1}{2}M_\delta\|\Theta_\alpha\|_{\mathcal{L}(\VV)}\delta) \|r^1\|_\VV^2. 
\end{align*}
In summary, using the definitions of $\theta_0$ and $\theta_\angle$, and setting $\wt{\gamma}=\gamma_\alpha -\tfrac{1}{2}M_\delta\|\Theta_\alpha\|_{\mathcal{L}(\VV)}\delta$,
\begin{align*}
\dua{&\opK_\KP(t_*^1 + r^1,\lambda)}{\Theta_\alpha r^1}\\
&\ge \wt{\gamma}\|r^1\|_\VV^2 - (\Delta(t_*^1)+M_\delta\delta)\|t_*^\angle + r^\angle\|_\VV\|\Pi_0\Theta_\alpha r^1\|_\VV\\
&\ge \wt{\gamma}\|r^1\|_\VV^2 - \frac{\Delta(t_*^1)+M_\delta\delta}{2}(\|t_*^\angle + r^\angle\|_\VV^2+\|\Pi_0\Theta_\alpha r^1\|_\VV^2)\\
&\ge (1-g)\wt{\gamma}(\|r^0\|_\VV^2+\|r^\angle\|_\VV^2) - \frac{\Delta(t_*^1)+M_\delta\delta}{2}
\left(\kappa^2 + 2\kappa\|r^\angle\|_\VV + \|r^\angle\|_\VV^2+\|\Pi_0\Theta_\alpha r^1\|_\VV^2\right)\\
&\ge \left((1-g)\wt{\gamma} - \frac{\Delta(t_*^1)+M_\delta\delta}{2}\Big(1+\max\{ \epsilon + 2(1+\epsilon^{-1})\theta_0, 2(1+\epsilon^{-1})\theta_\angle \}\Big) \right)\\
&\quad\times(\|r^0\|_\VV^2+\|r^\angle\|_\VV^2)- \frac{\Delta(t_*^1)+M_\delta\delta}{2}(\kappa^2 + \sqrt{2}\kappa(\delta-\kappa))\\
&\ge \left((1-g)\wt{\gamma} - \frac{\Delta(t_*^1)+M_\delta\delta}{2}\max\{ \epsilon + 2(1+\epsilon^{-1})\theta_0, 2(1+\epsilon^{-1})\theta_\angle \}\right) \frac{(\delta-\kappa)^2}{2} 
- \frac{\Delta(t_*^1)+M_\delta\delta}{2} \left(\kappa + 
\frac{\sqrt{2}}{2}(\delta-\kappa) \right)^2 .
\end{align*}
The positivity of the last expression follows from \cref{kappacond2}. We also used the bound
\begin{align*}
&\|r^\angle\|_\VV^2+\|\Pi_0\Theta_\alpha r^1\|_\VV^2\le (1+\epsilon^{-1})\|\Pi_0(\Theta_\alpha-I) r^1\|_\VV^2 + (1+\epsilon)\|r^0\|_\VV^2 + \|r^\angle\|_\VV^2\\
&\le 2(1+\epsilon^{-1})(\|\Pi_0(\Theta_\alpha-I)\Pi_0\|_{\mathcal{L}(\VV)}^2 \|r^0\|_\VV^2 + \|\Pi_0(\Theta_\alpha-I)\Pi_\angle\|_{\mathcal{L}(\VV)}^2 \|r^\angle\|_\VV^2) + (1+\epsilon)\|r^0\|_\VV^2 + \|r^\angle\|_\VV^2\\
&\le (1+\max\{ \epsilon + 2(1+\epsilon^{-1})\theta_0, 2(1+\epsilon^{-1})\theta_\angle \}) (\|r^0\|_\VV^2 + \|r^\angle\|_\VV^2).
\end{align*}

We can now apply the homotopy invariance of the degree to get $\deg(\opK_{\KP}(\cdot,\lambda),D,0)\equiv d\neq 0$ for all $\lambda\in[0,1]$, with $\delta$ decreased if necessary.

\section{A short proof of the Kowalski--Piecuch theorem}\label{kpapp}
To close this section, we present the main theorem\footnote{They call it the ``Fundamental Theorem of $\beta$-Nested Equation Formalism''.} of Kowalski and Piecuch \cite{piecuch2000search} in a somewhat different form. 
Define the \emph{KP energy} as
$$
\ene_\KP(t^1,\lambda)=\dua{\ham_K(t^0+\lambda t^\angle)\Phi_0}{\Phi_0}=\dua{\ham_K e^{T^0+\lambda T^\angle}\Phi_0}{\Phi_0},
$$
so that $\ene_\KP(t^1,0)=\ene_\cc(t^0)$ and $\ene_\KP(t^1,1)=\ene_\cc(t^1)$. If $\rho\ge 2$, then according to~(2.10) of Part I, we simply have
\begin{equation}\label{kpenesd}
\ene_\KP(t^1,\lambda)=\dua{\ham_K(t^0)\Phi_0}{\Phi_0}=\dua{\ham_K e^{T^0}\Phi_0}{\Phi_0},
\end{equation}
although we will not exploit this property in the proof.

\begin{theorem}[Kowalski--Piecuch]\label{kpthm}
\leavevmode
\begin{enumerate}[label=\normalfont(\Roman*)]
\item Suppose that $\Psi=(c_0 I + C^0)\Phi_0\in\HC^1_K$, with $c_0\in\RR$ and $c^0\in\VV^0$, satisfies the (weak) Schr\"odinger equation
\begin{equation}\label{kpschr}
\dua{\ham_K\Psi}{\Phi}=\ene\dua{\Psi}{\Phi}\quad\text{for all}\quad \Phi\in\HC^1_K
\end{equation}
for some $\ene\in\RR$ and that
\begin{equation}\label{kpcond}
\dua{e^{T^0_{**}(\lambda)+\lambda T_{**}^\angle(\lambda)} \Phi_0}{\Psi}\neq 0,
\end{equation}
where $t_{**}^1(\lambda)\in\VV^1$ is a zero of $\opK_\KP(\cdot,\lambda)$ for \emph{all} $\lambda\in[0,1]$.
Then 
$$
\ene_\KP(t_{**}^1(\lambda),\lambda)\equiv\ene\quad\text{for all}\quad\lambda\in[0,1].
$$
\item Suppose that $\Psi=(c_0 I + C^1)\Phi_0$, with $c_0\in\RR$ and $c^1\in\VV^1$, satisfies \cref{kpschr} for some $\ene\in\RR$. If \cref{kpcond} holds true for \emph{some} $\lambda\in[0,1]$ and $t_{**}^1=t_{**}^1(\lambda)\in\VV^1$ with $\opK_\KP(t_{**}^1,\lambda)=0$, then
\begin{align*}
\ene_\KP(t_{**}^1(\lambda),\lambda)-\ene&=\frac{\dua{(\ham_K(t^1_{**}) - \ham_K (t^0_{**}+\lambda t^\angle_{**}))\Phi_0}{ \Pi_{\VC^\angle}(e^{T^0_{**}+\lambda T^\angle_{**}})^\dag  C^\angle\Phi_0}}{\dua{e^{T^0_{**}+\lambda T_{**}^\angle} \Phi_0}{\Psi}}\\
&= (1-\lambda)\frac{\dua{\Gamma(t_{**}^1,\lambda)\Phi_0}{\Pi_{\VC^\angle}(e^{T^0_{**}+\lambda T^\angle_{**}})^\dag C^\angle\Phi_0}}{\dua{e^{T^0_{**}+\lambda T_{**}^\angle} \Phi_0}{\Psi}},
\end{align*}
where $\Gamma(t^1,\lambda)$ is given by \cref{gammadef} below.
\end{enumerate}
\end{theorem}

Furthermore, in case the energy blows up, we have the following.

\begin{theorem}\label{kpblowup}
Suppose that $\Psi=(c_0 I + C^1)\Phi_0$, with $c_0\in\RR$ and $c^1\in\VV^1$, satisfies \cref{kpschr} for some $\ene\in\RR$.
Furthermore, assume that $t_{**}^1(\lambda)\in\VV^1$ is a zero of $\opK_\KP(\cdot,\lambda)$ for all $\lambda$ in a neighborhood of some $\lambda_0\in[0,1]$.
If $|\ene_\KP(t_{**}^1(\lambda),\lambda)|\to\infty$ as $\lambda\to \lambda_0$ and
\begin{equation}\label{kpgammacond}
\dua{\Gamma(t_{**}^1(\lambda),\lambda)\Phi_0}{\Pi_{\VC^\angle}(e^{T^0_{**}+\lambda T^\angle_{**}})^\dag C^\angle\Phi_0}=\mathcal{O}\left(\frac{1}{1-\lambda}\right)\quad (\lambda\to \lambda_0),
\end{equation}
then
$$
\dua{e^{T^0_{**}(\lambda)+\lambda T_{**}^\angle(\lambda)} \Phi_0}{\Psi}\to 0 \quad (\lambda\to \lambda_0).
$$
\end{theorem}

Part (I) of \cref{kpthm} says that if one can solve the Schr\"odinger equation \emph{exactly} on $\VV^0$ for an eigenvalue $\ene$ and \cref{kpcond} holds true for
a zero $t_{**}^1(\lambda)\in\VV^1$  of $\opK_\KP(\cdot,\lambda)$ for all $\lambda$,
then the KP energy $\ene_\KP(t_{**}^1(\lambda),\lambda)$ is identically $\ene$. Notice that $t_{**}^1(1)$ is not required to represent $\Psi$, i.e. $e^{T_{**}^1(1)}\Phi_0\neq \Psi$
is allowed. Also, no regularity of the trajectory $\lambda\mapsto t_{**}^1(\lambda)$ is demanded.

Part (II) stipulates that the Schr\"odinger equation can be solved on $\VV^1$ with an eigenvalue $\ene$ and that the nonorthogonality condition \cref{kpcond} holds true for some $t_{**}^1\in\VV^1$ zero of $\opK_\KP(\cdot,\lambda)$ for some $\lambda$. Then, the error in the energy can be expressed by the stated formula.
If one assumes the hypothesis for all $\lambda\in[0,1]$ in a neighborhood of $1$, then we can conclude that
 the KP energy $\ene_\KP(t_{**}^1(\lambda),\lambda)$ tends to $\ene$ smoothly, as $\lambda\to 1$. Again, no regularity of $\lambda\mapsto t_{**}^1(\lambda)$ is needed.

Finally, \cref{kpblowup} considers the case when the KP energy diverges as $\lambda\to \lambda_0$.
Assuming the growth condition \cref{kpgammacond}, we can conclude that the KP solution $e^{T^0_{**}(\lambda)+\lambda T_{**}^\angle(\lambda)}\Phi_0$
becomes orthogonal to the eigenstate $\Psi$.

For the proof, we need the following lemma which recasts the KP equations in an ``unlinked'' form.
\begin{lemma}\label{kpunlink}
Suppose that $t_{**}^1=t_{**}^1(\lambda)\in\VV^1$ is such that $\opK_{\KP}(t^1_{**},\lambda)=0$ for some $\lambda\in[0,1]$. Then,
\begin{equation}\label{kpeq}
\dua{(\ham_K - \ene_\KP(t_{**}^1,\lambda))e^{T^0_{**}+\lambda T_{**}^\angle} \Phi_0}{ S^1\Phi_0} = - \dua{\mathcal{G}(t_{**}^1,\lambda)\Phi_0}{\Pi_{\VC^\angle}(e^{T^0_{**}+\lambda T^\angle_{**}})^\dag S^\angle\Phi_0},
\end{equation}
where
$$
\mathcal{G}(t_{**}^1,\lambda)=\ham_K(t^1_{**}) - \ham_K (t^0_{**}+\lambda t^\angle_{**}).
$$
Moreover, $\mathcal{G}(t_{**}^1,\lambda)=(1-\lambda)\Gamma(t_{**}^1,\lambda)$, where
\begin{equation}\label{gammadef}
\Gamma(t_{**}^1,\lambda)=\sum_{k=1}^{2N} \frac{(1-\lambda)^{k-1}}{k!} e^{-(T_{**}^0+\lambda T_{**}^\angle)} [\ham_K, T_{**}^\angle]_{(k)} e^{T_{**}^0+\lambda T_{**}^\angle} .
\end{equation}
\end{lemma}
\begin{proof}
Assume that $\opK_{\KP}(t^1_{**},\lambda)=0$, so using \cref{kpcomp} we have
$$
\dua{\ham_K(t_{**}^0+\lambda t_{**}^\angle) \Phi_0}{S^0 \Phi_0}+\dua{\ham_K(t_{**}^1) \Phi_0 }{S^\angle \Phi_0}=0\quad\text{for all}\quad s^1 \in\VV^1.
$$ 
This can be rewritten as
$$
\dua{\ham_K(t_{**}^0+\lambda t_{**}^\angle) \Phi_0}{S^1\Phi_0 } = 
-\dua{(\ham_K(t_{**}^1)-\ham_K(t_{**}^0+\lambda t_{**}^\angle) ) \Phi_0}{S^\angle  \Phi_0}\quad\text{for all}\quad s^1\in\VV^1,
$$
or,
$$
\dua{\ham_K(t_{**}^0+\lambda t_{**}^\angle) \Phi_0 }{S^1 \Phi_0}=-\dua{\mathcal{G}(t_{**}^1,\lambda) \Phi_0 }{S^\angle \Phi_0}\quad\text{for all}\quad s^1\in\VV^1.
$$
Then we can write
\begin{align*}
\dua{(\ham_K - \ene_\KP(t_{**}^1,\lambda))e^{T^0_{**}+\lambda T^\angle_{**}}\Phi_0}{ S^1\Phi_0}
&=\dua{e^{-(T^0_{**}+\lambda T^\angle_{**})}(\ham_K - \ene_\KP(t_{**}^1,\lambda))e^{T^0_{**}+\lambda T^\angle_{**}}\Phi_0}{ (e^{T^0_{**}+\lambda T^\angle_{**}})^\dag S^1\Phi_0} \\
&=\dua{e^{-(T^0_{**}+\lambda T^\angle_{**})}\ham_K e^{T^0_{**}+\lambda T^\angle_{**}}\Phi_0}{ \Pi_{\VC^1}(e^{T^0_{**}+\lambda T^\angle_{**}})^\dag S^1\Phi_0}\\
&=-\dua{\mathcal{G}(t_{**}^1,\lambda)\Phi_0}{\Pi_{\VC^\angle}(e^{T^0_{**}+\lambda T^\angle_{**}})^\dag S^1\Phi_0}\\
&=-\dua{\mathcal{G}(t_{**}^1,\lambda)\Phi_0}{\Pi_{\VC^\angle}(e^{T^0_{**}+\lambda T^\angle_{**}})^\dag S^\angle\Phi_0}
\end{align*}
for all $s^1\in\VV^1$. In the last step we used that $\Pi_{\VC^\angle}(e^{T^0_{**}+\lambda T^\angle_{**}})^\dag$ maps $\VC^0$ to zero.
The ``moreover'' part is a simple expansion using \cref{bchdoub},
\begin{align*}
\mathcal{G}(t_{**}^1,\lambda)&=e^{-(T^0_{**}+\lambda T^\angle_{**})} (e^{-(1-\lambda) T^\angle_{**}}\ham_K e^{(1-\lambda) T^\angle_{**}} - \ham_K) e^{T^0_{**}+\lambda T^\angle_{**}}\\
&=\sum_{k=1}^{2N} \frac{(1-\lambda)^k}{k!} e^{-(T^0_{**}+\lambda T^\angle_{**})}[\ham_K, T^\angle_{**}]_{(k)} e^{T^0_{**}+\lambda T^\angle_{**}}.
\end{align*}
\end{proof}

\begin{proof}[Proof of \cref{kpthm}]
We have using \cref{kpunlink}, 
\begin{align*}
0&=\dua{(\ham_K - \ene_\KP(t_{**}^1(\lambda),\lambda))e^{T^0_{**}(\lambda)+\lambda T_{**}^\angle(\lambda)} \Phi_0}{(c_0I +C^0)\Phi_0}\\
&=(\ene- \ene_\KP(t_{**}^1(\lambda),\lambda)) \dua{e^{T^0_{**}(\lambda)+\lambda T_{**}^\angle(\lambda)} \Phi_0}{\Psi}.
\end{align*}
Similarly, for part (II)
\begin{align*}
\dua{(\ham_K(t^1_*) - \ham_K (t^0_{**}+\lambda t^\angle_{**}))\Phi_0}{ \Pi_{\VC^\angle}(e^{T^0_{**}+\lambda T^\angle_{**}})^\dag C^\angle\Phi_0}
&=\dua{(\ham_K - \ene_\KP(t_{**}^1(\lambda),\lambda))e^{T^0_{**}+\lambda T_{**}^\angle} \Phi_0}{(c_0I +C^0 + C^\angle)\Phi_0}\\
&=(\ene- \ene_\KP(t_{**}^1(\lambda),\lambda)) \dua{e^{T^0_{**}+\lambda T_{**}^\angle} \Phi_0}{\Psi}.
\end{align*}
\Cref{kpblowup} also follows from the previous equality.
\end{proof}

%%%%%%%%%%%%%%%%%%%%%%%%%%%%%%%%%%%%%%%%%%%%%%%%%%%%%%%%%%%%%%%%%%%%%%%%%%%%%%%%%%%%%%%%%%%%%%%%%%%
%% ACKNOWLEDGEMENTS
%%%%%%%%%%%%%%%%%%%%%%%%%%%%%%%%%%%%%%%%%%%%%%%%%%%%%%%%%%%%%%%%%%%%%%%%%%%%%%%%%%%%%%%%%%%%%%%%%%%

\noindent\emph{Acknowledgements.}
The authors would like to thank Fabian M. Faulstich and Simen Kvaal for helpful discussions and comments on the manuscript. 
The useful suggestions of the anonymous reviewer are gratefully acknowledged. This work has received funding from the Norwegian Research Council through Grant Nos. 287906 (CCerror) and 262695 (CoE Hylleraas Center for Quantum Molecular Sciences).

\bibliographystyle{abbrv}
\bibliography{cc}

\end{document}